\documentclass[final]{siamltex}

\usepackage{amssymb,amsmath,algorithm2e,graphicx}

\usepackage[usenames]{color}

\renewcommand{\b}[1]{\boldsymbol{#1}}
\newcommand{\bn}{{\boldsymbol{n}}}
\newcommand{\bq}{{\boldsymbol{q}}}
\newcommand{\bw}{{\boldsymbol{w}}}
\newcommand{\cA}{\mathcal{A}}
\newcommand{\cB}{\mathcal{B}}
\newcommand{\cF}{\mathcal{F}}
\newcommand{\CF}{C_\mathrm{F}}
\newcommand{\CFavg}{C_\mathrm{F}^\mathrm{avg}}
\newcommand{\CFlow}{C_\mathrm{F}^\mathrm{low}}
\newcommand{\CFup}{C_\mathrm{F}^\mathrm{up}}
\newcommand{\cM}{\mathcal{M}}
\newcommand{\CP}{C_\mathrm{P}}
\newcommand{\CPlow}{C_\mathrm{P}^\mathrm{low}}
\newcommand{\CPup}{C_\mathrm{P}^\mathrm{up}}
\newcommand{\cT}{\mathcal{T}}
\newcommand{\CT}{C_\mathrm{T}}
\newcommand{\CTlow}{C_\mathrm{T}^\mathrm{low}}
\newcommand{\CTup}{C_\mathrm{T}^\mathrm{up}}
\newcommand{\ddiv}{\operatorname{div}}

\newcommand{\ErrTol}{E_{\mathrm{TOL}}} 
\newcommand{\GammaD}{{\Gamma_{\mathrm{D}}}}
\newcommand{\GammaN}{{\Gamma_{\mathrm{N}}}}
\newcommand{\Hdiv}[1][\Omega]{\mathbf{H}(\ddiv,#1)}
\newcommand{\lil}{{\lambda_1^{\mathrm{low}}}}
\newcommand{\liil}{{\lambda_2^{\mathrm{low}}}}
\newcommand{\liiu}{{\lambda_2^{\mathrm{up}}}}
\newcommand{\Nel}{N_\mathrm{el}}
\newcommand{\NDOF}{N_\mathrm{DOF}}
\newcommand{\norm}[1]{\|#1\|}
\newcommand{\oGammaD}{\overline\Gamma_{\mathrm{D}}}
\newcommand{\oGammaN}{\overline\Gamma_{\mathrm{N}}}
\newcommand{\oHI}{\widetilde{H}{}^1(\Omega)}
\newcommand{\R}{\mathbb{R}}
\newcommand{\RelErr}{E_{\mathrm{REL}}} 
\newcommand{\trinorm}[1]{\|#1\|_a}
\newcommand{\ttg}[1]{\textrm{(\ref{#1})}}

\title{Two-sided bounds for eigenvalues of differential operators
with applications to Friedrichs', Poincar\'e, trace, and similar constants\thanks{The support of I.\ \v{S}ebestov\'a by the project MathMAC - University center for mathematical modeling, applied analysis and computational mathematics of the Charles University in Prague and the support of T.\ Vejchodsk\'y by RVO~67985840
are gratefully acknowledged.}}

\author{Ivana \v{S}ebestov\'a\footnotemark[2]
        \and Tom\'a\v s Vejchodsk\'y\footnotemark[3]}

\begin{document}

\maketitle

\renewcommand{\thefootnote}{\fnsymbol{footnote}}

\footnotetext[2]{Department of Numerical Mathematics, Faculty of Mathematics and Physics, Charles University in Prague, Sokolovsk\'a 83, CZ-186 75 Praha 8, Czech Republic ({\tt ivasebestova@seznam.cz}).}
\footnotetext[3]{Institute of Mathematics, Academy of Sciences, {\v Z}itn{\'a} 25, CZ-115 67 Praha 1, Czech Republic ({\tt vejchod@math.cas.cz}).}

\renewcommand{\thefootnote}{\arabic{footnote}}

\begin{abstract}
We present a general numerical method for computing guaranteed two-sided
bounds for principal eigenvalues of symmetric linear elliptic
differential operators.
The approach is based on the Galerkin method, on the method
of a priori-a posteriori inequalities,
and on a complementarity technique.
The two-sided bounds are formulated in a general Hilbert space setting
and as a byproduct we prove an abstract inequality of Friedrichs'--Poincar\'e
type.
The abstract results are then applied to Friedrichs', Poincar\'e,
and trace inequalities and fully computable two-sided bounds on
the optimal constants in these inequalities are obtained.
Accuracy of the method is illustrated on numerical examples.
\end{abstract}



\pagestyle{myheadings}
\thispagestyle{plain}
\markboth{IVANA \v{S}EBESTOV\'A  AND TOM\'A\v{S} VEJCHODSK\'Y}{TWO-SIDED BOUNDS FOR EIGENVALUES}


\section{Introduction}

%

Eigenvalue problems for differential operators have attracted a lot of
attention as they have many applications. These include the dynamic analysis of
mechanical systems
\cite{Alonso_Russo_Padra_Rodri_AEE_FEM_SAVP_01,
Chu_King_Tsung_Palindr_EP_sum_2010,Leissa_Vibrat_Plates-69}, linear stability of
flows in fluid mechanics \cite{Larson_Instab_Visco_Flows_92}, and electronic
band structure calculations \cite{Martin_Electr_struc_bas_th_prac_methods_04}.
In this paper, we concentrate on guaranteed two-sided bounds of the principal
(smallest) eigenvalue of symmetric linear elliptic operators. The standard Galerkin method for solution of eigenproblems is efficient
and its convergence and other properties are well analysed
\cite{BabOsb:1989,BabOsb:1991,Boffi:2010}.
It is also well known for providing upper bounds on eigenvalues.
However, in many applications a reliable \emph{lower bound} of the smallest
eigenvalue is the key piece of information
and, unfortunately, the Galerkin method cannot provide it.

The question of lower bounds on the smallest eigenvalue has already been
studied for several decades. For example see
\cite{Ogawa_Protter_A_low_bound_first_eig_sec_order_ellip_oper_68}, where the
second order elliptic eigenvalue problems with Dirichlet boundary conditions are
considered. Another technique that gives the lower bounds not only for the first
eigenvalue is the method of intermediate problems. It is based on finding
a base problem and subsequently introducing intermediate problems that
give lower bounds for eigenvalues of the original problem and at the same time
can be resolved explicitly, see for example
\cite{Bazley_Fox_Low_bounds_eig_Schrodinger_eq_61,
Beattie_Goerisch_Methods_low_bounds_eig_self-adjoint_95,
Wein_Stenger_Methods_Intermed_prob_eig_72}. A survey of this
technique can be found in
\cite{Fox_Rheinboldt_Comput_meth_deter_low_bounds_eig_oper_Hilbert_66}.


Nonconforming methods have been used for computing lower bounds on
eigenvalues, see for example \cite{AndRac:2012,LinXieLuoLiYan:2010,LuoLinXie:2012,Rannacher:1979,YanZhaLin:2010}.
However, these lower bounds are valid asymptotically only and hence
these methods \emph{do not guarantee} that the computed approximation
is really below the exact value.
Recently, sufficient conditions for producing lower bounds for eigenvalues of symmetric elliptic operators by nonconforming methods have been provided
in \cite{Hu_Huang_Lin_Lower_bounds_EO_NFEMs_2011arXiv}.
The described technique
is based on satisfying the saturation assumption and on the condition saying that the local approximation property of the underlying finite element space have to be better than its global continuity property. It is proved that the second condition is met by most commonly used nonconforming methods such as the Wilson element, linear nonconforming Crouzeix--Raviart element, the enriched nonconforming rotated $Q_1$ element, the Morley element, the Adini element, and the Morley--Wang--Xu element. The saturation assumption is proved for the Morley--Wang--Xu element, the enriched Crouzeix--Raviart element, and the Wilson element. Furthermore, new nonconforming
methods satisfying these properties are proposed. However, no numerical experiments
are presented.

Further, let us point out a recent result \cite{LinXieXu:MatComp},
where two-sided a priori bounds for the discretization error of eigenvectors
are given.

The method of \emph{a priori-a posteriori inequalities} that can be used
for computation of lower bounds on eigenvalues was described and published
in \cite{KutSig:1978,KutSig:1984,Sigillito:1977}. However, in these original
publications $C^2$-continuous test and trial functions have been used in
order to
compute the actual lower bound. These functions are difficult to work with
and therefore, we couple the original method of a priori-a posteriori inequalities
with the complementarity technique, where a certain flux function has to be
reconstructed, see, e.g.,
\cite{robustaee:2010,CheFucPriVoh:2009,HasHla:1976,Rep:2008,Complement:2010,Vohralik:2011}.
This flux reconstruction can be done in many ways.
We choose the most straightforward approach that can be handled by standard
Raviart--Thomas--N\'{e}d\'{e}lec finite element method.

Moreover, we generalize the original method of a priori-a posteriori inequalities
to the case of a compact operator between a pair of Hilbert spaces.
This generalization is especially useful for computing two-sided bounds
of the optimal constant in Friedrichs', Poincar\'e, trace, and similar
inequalities.
It is based on the fact that
the optimal constant in these inequalities is inversely proportional
to the square root of the smallest eigenvalue of the corresponding
symmetric linear elliptic partial differential operator.

Further, the generalization we have made enables to set up an abstract framework
in the Hilbert space setting. The abstract results then easily apply to
symmetric linear elliptic partial differential operators
and consequently to the optimal constant in the inequalities of
Friedrichs'--Poincar\'e type. Furthermore, as a byproduct of the abstract
setting, we obtain a simple proof of the validity of an abstract
inequality of Friedrichs'--Poincar\'e type in the Hilbert space setting.
The particular choices of the pair of Hilbert spaces, corresponding scalar products,
and the compact operator then naturally yield the validity of
Friedrichs', Poincar\'e, trace, Korn's and other inequalities.
%

The main motivation for our interest in two-sided bounds of the constants
in Friedrichs', Poincar\'e, trace, and similar inequalities stems
from the need of these bounds in a posteriori error estimation for numerical
solutions of partial differential equations. In particular, the existing
guaranteed upper bounds on the energy norm of the error utilize a kind of
complementarity technique, see, e.g.,
\cite{robustaee:2010,BraSch:2008,Korotov:2007,Rep:2008,Complement:2010,Vohralik:2011}.
Estimates of this kind contain constants from Friedrichs', Poincar\'e,
trace, and similar inequalities. Optimal values of these constants
are often unknown and therefore suitable approximations have to be used
in the error estimates.
These approximations have to provide upper bounds on these constants
in order to guarantee that the total error estimator is an
upper bound on the error.
Moreover, they have to be accurate due to the accuracy and
efficiency of the error estimates.

The method presented in this paper provides accurate upper bounds on these
constants. In addition, this method naturally considers the dependence
of the optimal constants on the equation coefficients and on the boundary
conditions. This dependence can be strong \cite{AM2012}
and its capturing might be
crucial for accuracy and robustness of a posteriori error bounds.

The rest of the paper is organized as follows.
Section~\ref{se:abseigenprob} introduces a general
variational eigenvalue problem in the Hilbert space setting.
It uses the spectral theory of compact operators to prove several properties
of this eigenvalue problem including the existence of the principal eigenvalue.
In Section~\ref{se:absineq} we naturally prove the abstract inequality
of Friedrichs'--Poincar\'e type and show the relation between the optimal constant
and the principal eigenvalue. Further, we briefly describe the Galerkin
method that yields an upper bound on the principal eigenvalue
and concentrate on the method of a priori-a posteriori inequalities
and on an abstract complementarity result leading to a lower bound
on the principal eigenvalue.
Sections~\ref{se:Fried}--\ref{se:trace} apply the abstract results
to the case of Friedrichs', Poincar\'e, and trace inequality
and fully computable two-sided bounds on the optimal constants in these
inequalities are obtained. Presented numerical experiments illustrate
accuracy of the method and dependence of the optimal constants on
a nonhomogeneous diffusion parameter.
Finally, Section~\ref{se:concl} draws the conclusions.

\section{Variational eigenvalue problem in the Hilbert space setting}
\label{se:abseigenprob}

Let $V$ and $H$ be two real Hilbert spaces with scalar products $(\cdot,\cdot)_V$
and $(\cdot,\cdot)_H$, respectively.
The norms induced by these scalar products are denoted by $\norm{\cdot}_V$
and $\norm{\cdot}_H$.
Further, let $\gamma : V \rightarrow H$ be a continuous, linear, and \emph{compact}
operator.
The center of our interest is the following eigenvalue problem.
Find $\lambda_i \in \R$, $u_i \in V$, $u_i \neq 0$ such that
\begin{equation}
\label{eq:EP1}
(u_i,v)_V = \lambda_i (\gamma u_i, \gamma v)_H \quad\forall v \in V.
\end{equation}
First, let us show that eigenvalues $\lambda_i$ of \ttg{eq:EP1} are positive.

\begin{lemma}
\label{le:lambdapos}
If $u_i \in V$ is an eigenvector corresponding to an eigenvalue
$\lambda_i$ of \ttg{eq:EP1} then $\gamma u_i \neq 0$ and $\lambda_i > 0$.
\end{lemma}
\begin{proof}
Since $u_i \neq 0$, we have by \ttg{eq:EP1} that
$0 \neq \norm{u_i}_V^2 = \lambda_i \norm{\gamma u_i}_H^2$.
Thus, $\gamma u_i \neq 0$ and $\lambda_i$ has to be positive.
\end{proof}

Below we show that eigenvalues $\lambda_i$ and eigenvectors $u_i$ of \ttg{eq:EP1} correspond to eigenvalues and eigenvectors of a compact operator, respectively. Consequently, these eigenvalues form a countable sequence that can be ordered as $\lambda_1 \leq \lambda_2 \leq \cdots$.
To show this correspondence, we define a
\emph{solution operator} $S : H \rightarrow V$.
If $f\in H$ is arbitrary then the mapping $v \mapsto (f, \gamma v)_H$
is a continuous linear form on $V$ and, hence, the Riesz representation
theorem yields existence of a unique element $Sf \in V$ such that
\begin{equation}
\label{eq:solop}
  (Sf, v)_V = (f, \gamma v)_H \quad\forall v\in V.
\end{equation}
Consequently, the solution operator $S$ is linear and continuous.

The composition of operators $S$ and $\gamma$ is a
linear, continuous, and compact operator
$S\gamma : V \rightarrow V$, see \cite[Theorem~4.18 (f)]{Rudin_FA_91}.
In addition, this operator is selfadjoint, because
definition \ttg{eq:solop} yields
\begin{equation} \label{selfadjoint}
  (S\gamma u, v)_V = (\gamma u, \gamma v)_H = (\gamma v, \gamma u)_H
= (S\gamma v, u)_V = (u, S\gamma v)_V.
\end{equation}
Therefore, we can use the Hilbert--Schmidt spectral theorem for $S\gamma$,
see \cite[Theorem~4, Chapter II, Section 3]{Gaal_Lin_anal_repres_theo_73} and obtain that $V$ can be decomposed
into a direct sum of two subspaces
\begin{equation}
\label{eq:Vdecomp}
  V = \cM \oplus \ker(S\gamma),
\end{equation}
where $\ker(S\gamma) = \{ v \in V : S\gamma v = 0 \}$ is the kernel of $S\gamma$
and $\cM$ is generated by all eigenvectors of the operator $S\gamma$
corresponding to nonzero eigenvalues.
Let us recall that $u_i \in V$, $u_i \neq 0$ is an eigenvector of $S\gamma$
corresponding to an eigenvalue $\mu_i \in \R$ if
\begin{equation}
\label{eq:EP2}
  S\gamma u_i = \mu_i u_i.
\end{equation}
Furthermore, the Hilbert--Schmidt spectral theorem implies that
the system $u_1, u_2, \dots$ of eigenvectors corresponding to nonzero
eigenvalues of \ttg{eq:EP2} is countable and orthogonal in $V$.
A simple consequence of \ttg{eq:solop} and \ttg{eq:EP2} is the orthogonality
of images $\gamma u_i$, $i=1,2,\dots$ in $H$.
In this paper we consider the following normalization of these eigenvectors:
\begin{equation}
\label{eq:normalization}
  (\gamma u_i, \gamma u_j)_H = \delta_{ij} \quad\forall i,j=1,2,\dots,
\end{equation}
where $\delta_{ij}$ stands for the Kronecker's delta.

Now, let us observe that eigenproblems \ttg{eq:EP1} and \ttg{eq:EP2}
correspond to each other and, therefore, the spectral properties
of the compact operator $S\gamma$ translate to the properties
of the eigenproblem \ttg{eq:EP1}.

\begin{lemma}
\label{le:Sgamma}
Considering the above setting, the following statements hold true.
\begin{enumerate}
\item[1.]
Number $\lambda_i \in \R$ is an eigenvalue corresponding
to the eigenvector $u_i \in V$ of \ttg{eq:EP1}
if and only if
$\mu_i = 1/\lambda_i$ is a nonzero eigenvalue corresponding
to the eigenvector $u_i$ of the operator $S\gamma$, see \ttg{eq:EP2}.
\item[2.]
The number of eigenvalues $\lambda_i$ of \ttg{eq:EP1} such that $\lambda_i \leq M$
is finite for any $M>0$.
\item[3.]
The value $\lambda_1 = \inf\limits_{u\in V, u\neq 0}
{\norm{u}_V^2}/{\norm{\gamma u}_H^2}
$
is the smallest eigenvalue of \ttg{eq:EP1}.
\end{enumerate}
\end{lemma}
\begin{proof}
1.\ Definition \ttg{eq:solop}
yields identity $(\gamma u_i,\gamma v)_H = ( S\gamma u_i, v)_V$
for all $v \in V$. Hence, the equality \ttg{eq:EP1} can be rewritten as
$(u_i,v)_V = \lambda_i (S\gamma u_i,v)_V$, which is equivalent
to \ttg{eq:EP2} with $\mu_i = 1/\lambda_i$ provided that $\lambda_i \neq 0$
and $\mu_i \neq 0$.
Since Lemma~\ref{le:lambdapos} guarantees $\lambda_i > 0$ for all $i=1,2,\dots$,
the only condition is $\mu_i \neq 0$.

2.\ If we denote the spectrum of $S\gamma$ by $\sigma(S\gamma)$
then the compactness of $S\gamma$ implies that the set
$[\varepsilon,\infty) \cap \sigma(S\gamma)$ is finite
for any $\varepsilon > 0$, see
\cite[Theorem~4.24 (b)]{Rudin_FA_91}.
The statement follows immediately from the fact that
$\lambda_i = 1/\mu_i$ for $\mu_i \neq 0$.

3.\ Since $S\gamma$ is selfadjoint, see~\ttg{selfadjoint},
the Courant--Fischer--Weyl min-max principle, see, e.g.,~\cite{Strang_Fix_analysis_FEM_08},
implies that
$$
  \mu_1 = \sup \{ (S\gamma v,v)_V : \norm{v}_V = 1 \}
  = \sup\limits_{v\in V, v\neq 0} \frac{(S\gamma v, v)_V}{\norm{v}_V^2}
  = \sup\limits_{v\in V, v\neq 0} \frac{\norm{\gamma v}_H^2}{\norm{v}_V^2}
$$
is finite and it is the largest eigenvalue of the operator $S\gamma$.
Consequently,
\begin{equation}
\label{eq:lambda1inf}
  \lambda_1 = \mu_1^{-1}
= \left( \sup\limits_{v\in V, v\neq 0} \frac{\norm{\gamma v}_H^2}{\norm{v}_V^2}
  \right)^{-1}
= \inf\limits_{v\in V, v\neq 0} \frac{\norm{v}_V^2}{\norm{\gamma v}_H^2}
\end{equation}
is the smallest eigenvalue of problem \ttg{eq:EP1}.
\end{proof}

\section{Abstract inequality of Friedrichs'--Poincar\'e type}
\label{se:absineq}

\subsection{The proof of the abstract inequality}
Properties of the eigenproblem \ttg{eq:EP1} can be utilized in a simple
way to derive an abstract inequality of Friedrichs'--Poincar\'e type.
The Hilbert space versions of the particular Friedrichs', Poincar\'e,
trace, Korn's and similar inequalities easily follow from this
abstract result. For examples see Sections~\ref{se:Fried}--\ref{se:trace}.

\begin{theorem}[Abstract inequality]
\label{th:absineq}
Let $\gamma : V \rightarrow H$ be a continuous, linear, and compact
operator between Hilbert spaces $V$ and $H$.
Let $\lambda_1$ be the smallest eigenvalue of the problem \ttg{eq:EP1}.
Then
\begin{equation}
\label{eq:absineq}
  \norm{\gamma v}_H \leq C_\gamma \norm{v}_V \quad \forall v \in V
\end{equation}
with $C_\gamma = \lambda_1^{-1/2}$.
Moreover, this constant is optimal in the sense that it is the smallest
possible constant such that \ttg{eq:absineq} holds for all $v\in V$.
\end{theorem}
\begin{proof}
The validity of the abstract inequality follows immediately from \ttg{eq:lambda1inf}:
$$
  \norm{\gamma v}_H^2 \leq \lambda_1^{-1} \norm{v}_V^2  \quad\forall v \in V.
$$
This inequality holds as the equality for $v=u_1$ and thus, the constant
$C_\gamma = \lambda_1^{-1/2}$ is optimal.
\end{proof}

\subsection{Upper bound on the smallest eigenvalue}
\label{se:absupper}

The upper bound on $\lambda_1$ can be computed by the standard
Galerkin method, which is both accurate and efficient
\cite{BabOsb:1989,BabOsb:1991,Boffi:2010}. This method is based on the projection of the eigenproblem
\ttg{eq:EP1} into a finite dimensional subspace $V^h \subset V$.
We seek eigenvectors $u_i^h \in V^h$, $u_i^h \neq 0$, and eigenvalues
$\lambda_i^h$ such that
\begin{equation}
\label{eq:EP1h}
(u_i^h,v^h)_V = \lambda_i^h (\gamma u_i^h, \gamma v^h)_H \quad\forall v^h \in V^h.
\end{equation}
Let $\{\varphi_j\}_{j=1,\dots,N}$ be a basis of the space $V^h$. Then, we can formulate problem \ttg{eq:EP1h} equivalently as a generalized eigenvalue problem
$$
  A y_i = \lambda_i^h M y_i
$$
for matrices $A$ and $M$ with entries
$$
  A_{jk} = (\varphi_k,\varphi_j)_V
\quad\text{and}\quad
  M_{jk} = (\gamma \varphi_k, \gamma \varphi_j)_H,
\quad j,k=1,2,\dots,N.
$$
The eigenvectors $y_i \in \R^N$ and $u_i^h \in V^h$ are linked by the relation
$u_i^h = \sum_{j=1}^N y_{i,j} \varphi_j$.
The generalized matrix eigenvalue problem can be solved by efficient
methods of numerical linear algebra \cite{BaiDemDonRuhVor:2000}.

The Galerkin method for eigenvalue problems is very well understood.
The convergence and the speed of convergence of this method is established
for example in \cite{BabOsb:1989,BabOsb:1991,Boffi:2010}.
It is well known \cite{Boffi:2010} that the Galerkin method
approximates the exact eigenvalues from above, hence
$$
  \lambda_i \leq \lambda_i^h, \quad\forall i=1,2,\dots.
$$
In particular, the upper bound on the smallest eigenvalue $\lambda_1$
and the corresponding lower bound on the optimal constant $C_\gamma$ read
\begin{equation}
\label{eq:abslowb}
  \lambda_1 \leq \lambda_1^h
\quad\text{and}\quad
(\lambda_1^h)^{-1/2} \leq C_\gamma.
\end{equation}

\subsection{Lower bound on the smallest eigenvalue}
\label{se:lbound}

In this part we concentrate on a computable
lower bound on the smallest eigenvalue $\lambda_1$ of the problem~\ttg{eq:EP1}.
First, we formulate an auxiliary result, which states that the images $\gamma u_i$, $i=1,2,\dots$
of the orthogonal system of eigenvectors $u_i$ satisfy the Parseval's identity.
\begin{lemma}
\label{le:parseval}
Let $u_i$, $i=1,2,\dots$, be the above specified orthogonal system of
eigenvectors of the operator $S\gamma$ corresponding to nonzero eigenvalues.
Let these eigenvectors be normalized as in \ttg{eq:normalization}.
Let $u_* \in V$ be arbitrary.
Then
$$
  \norm{ \gamma u_* }_H^2
 =\sum\limits_{i=1}^{\infty} | (\gamma u_*, \gamma u_i)_H |^2.
$$
\end{lemma}
\begin{proof}
Due to the decomposition \ttg{eq:Vdecomp},
there exist unique components $u_*^\cM \in \cM$ and $u_*^0 \in \ker(S\gamma)$
such that $u_* = u_*^\cM + u_*^0$.
Since $S\gamma u_*^0 = 0$, we have $0 = (S\gamma u_*^0, u_*^0)_V =
(\gamma u_*^0, \gamma u_*^0)_H$ by \ttg{eq:solop}
and hence $\gamma u_*^0 = 0$.
Consequently, $\gamma u_* = \gamma u_*^\cM$.

System $\gamma u_i$, $i=1,2,\dots$, forms an orthonormal basis in $\gamma\cM$.
Thus, we can use the standard Parseval's identity in $\gamma\cM$ \cite[Theorem 2, Chapter III, Section 4]{Yosida_FA_94}
to obtain
$$
  \norm{ \gamma u_* }_H^2 = \norm{ \gamma u_*^\cM }_H^2
  = \sum\limits_{i=1}^{\infty} | (\gamma u_*^\cM, \gamma u_i)_H |^2
  = \sum\limits_{i=1}^{\infty} | (\gamma u_*, \gamma u_i)_H |^2.
$$
\end{proof}

The derivation of the lower bound on $\lambda_1$ is based on
the method of a priori-a posteriori inequalities.
This method relies on an abstract theorem proved in \cite{KutSig:1978}.
We formulate this theorem in the setting of this paper
and for the readers' convenience we present its brief proof.
Notice that in contrast to \cite{KutSig:1978}, Theorem~\ref{le:KutSig:1978}
operates with a pair of Hilbert spaces and with a compact operator
between them.
\begin{theorem}
\label{le:KutSig:1978}
Let $\gamma : V \rightarrow H$ be a continuous, linear, and compact
operator between Hilbert spaces $V$ and $H$.
Let $u_* \in V$ and $\lambda_* \in \R$ be arbitrary. Let $\lambda_i$, $i=1,2,\dots$, be eigenvalues of~\ttg{eq:EP1}.
Let us consider $w\in V$ such that
\begin{equation}
\label{eq:RW}
  (w,v)_V = (u_*,v)_V - \lambda_*(\gamma u_*, \gamma v)_H
  \quad \forall v \in V.
\end{equation}
If $\gamma u_* \neq 0$ then
\begin{equation}
\label{eq:KutSigest}
  \min\limits_{i} \left| \frac{\lambda_i - \lambda_*}{\lambda_i} \right|
  \leq \frac{\norm{\gamma w}_H}{\norm{\gamma u_*}_H}.
\end{equation}
\end{theorem}
\begin{proof}
Using Lemma~\ref{le:parseval} and definitions \ttg{eq:EP1} and \ttg{eq:RW}, we obtain
\begin{align*}
  \min\limits_{i} \left| \frac{\lambda_i - \lambda_*}{\lambda_i} \right|^2
  \norm{\gamma u_*}_H^2
&\leq \sum\limits_{i=1}^\infty \left| \frac{\lambda_i - \lambda_*}{\lambda_i}
  (\gamma u_*, \gamma u_i)_H \right|^2
\\
&= \sum\limits_{i=1}^\infty \left| \frac{(u_i,u_*)_V}{\lambda_i}
    - \frac{(u_* - w, u_i)_V}{\lambda_i} \right|^2
\\
&= \sum\limits_{i=1}^\infty \left| \frac{(w,u_i)_V}{\lambda_i} \right|^2
 = \sum\limits_{i=1}^\infty \left| (\gamma w,\gamma u_i)_H \right|^2
 = \norm{\gamma w}_H^2.
\end{align*}
\end{proof}

In order to obtain a computable lower bound on $\lambda_1$, we combine
the estimate \ttg{eq:KutSigest} with a complementarity technique,
see Sections~\ref{se:Fried}--\ref{se:trace}.
The bounds derived by the complementarity technique depend
on the particular choice of spaces $V$ and $H$,
but they have a common general structure.
The following theorem utilizes this general structure and presents
an abstract lower bound on $\lambda_1$.

\begin{theorem}[Abstract complementarity estimate]
\label{le:abscompl}
Let $u_*\in V$, $\lambda_* \in \R$ be arbitrary and let $w\in V$
satisfy \ttg{eq:RW}.
Let $\lambda_1$ be the smallest eigenvalue of \ttg{eq:EP1}
and
let the relatively closest eigenvalue to $\lambda_*$ be $\lambda_1$,
i.e. let
\begin{equation}
\label{eq:closest}
 \left|\frac{\lambda_1 - \lambda_*}{\lambda_1}\right|
\leq
  \left|\frac{\lambda_i - \lambda_*}{\lambda_i}\right|
\quad \forall i=1,2,\dots.
\end{equation}
Further, let $A \geq 0$ and $B \geq 0$ be such that
\begin{equation}
\label{eq:cond_B}
  B < \lambda_* \norm{\gamma u_*}_H
\end{equation}
and
\begin{equation}
\label{eq:abscompl}
  \norm{w}_V \leq A + C_\gamma B,
\end{equation}
where $C_\gamma$
is the optimal constant from \ttg{eq:absineq}.
Then
\begin{equation}
\label{eq:absbound}
  X_2^2 \leq \lambda_1
\quad\text{and}\quad
C_\gamma \leq 1/X_2,
\end{equation}
where
\begin{equation}
\label{eq:X2}
X_2 = \frac12 \left( -\alpha + \sqrt{\alpha^2 + 4(\lambda_* - \beta)} \right),
\quad \alpha=\frac{A}{\norm{\gamma u_*}_H}, \quad \text{and}
\quad \beta=\frac{B}{\norm{\gamma u_*}_H}.
\end{equation}
\end{theorem}
\begin{proof}
By using the fact that $C_\gamma = \lambda_1^{-1/2}$,
the estimate \ttg{eq:KutSigest}, assumptions of the theorem,
and the inequality \ttg{eq:absineq},
we obtain the validity of the following relation
$$
  \lambda_* C_\gamma^2 - 1
=
  \frac{\lambda_* - \lambda_1}{\lambda_1}
\leq
  \min\limits_i \left| \frac{\lambda_i - \lambda_*}{\lambda_i} \right|
\leq
  \frac{\norm{\gamma w}_H}{\norm{\gamma u_*}_H}
\leq
  C_\gamma \frac{\norm{w}_V}{\norm{\gamma u_*}_H}
\leq
  C_\gamma \alpha + C_\gamma^2 \beta.
$$
This is equivalent to the quadratic inequality
$$
  0 \leq C_\gamma^2(\beta-\lambda_*) + C_\gamma \alpha + 1.
$$
By solving it for $C_\gamma$ under the assumption \ttg{eq:cond_B},
we conclude that this inequality
is not satisfied for $C_\gamma > 1/X_2$.
Thus, $C_\gamma$ has to be at most $1/X_2$.
\end{proof}

The particular complementarity estimates have the form \ttg{eq:abscompl},
where the numbers $A$ and $B$ are obtained by an approximate minimization
procedure, see Sections~\ref{se:Fried}--\ref{se:trace}.
The better approximation of the exact minimizer we compute
the tighter bound \ttg{eq:abscompl} and consequently \ttg{eq:absbound}
we obtain. The exact minimizer yields equality in \ttg{eq:abscompl}
and $B = 0$. Therefore, the assumption \ttg{eq:cond_B} can always be satisfied by computing sufficiently accurate minimizer.

The assumption \ttg{eq:closest} is crucial and cannot be guaranteed unless
lower bounds on $\lambda_1$ and $\lambda_2$ are known.
However, since the Galerkin method is known to converge
\cite{BabOsb:1989,BabOsb:1991,Boffi:2010}
with a known speed, very accurate approximations of $\lambda_1$ and $\lambda_2$
can be computed. If these approximations are well separated, they can be used in \ttg{eq:closest} to verify its validity
with a good confidence.

In order to increase this confidence, we propose a test.
This test is based on the following observation.
Let $\lil \leq \lambda_1 \leq \lambda^* \leq \liil \leq \lambda_2 \leq \liiu$,
$D_1 = (\lambda_* - \lil)/\lil$ and $D_2 = (\liil-\lambda_*)/\liiu$.
Then inequality $D_1 \leq D_2$ implies the assumption \ttg{eq:closest}.
Indeed, if all these inequalities are satisfied then, clearly,
$\left| (\lambda_1 - \lambda_*)/\lambda_1 \right| \leq D_1
  \leq  D_2 \leq \left| (\lambda_2 - \lambda_*)/\lambda_2 \right|
  \leq |(\lambda_i - \lambda_*)/\lambda_i|$
 for all $i=2,3,\dots$ and \ttg{eq:closest} holds true.

Thus, if we knew guaranteed lower bounds $\lil$ and $\liil$
on the first eigenvalue $\lambda_1$ and the second eigenvalue $\lambda_2$
of  \ttg{eq:EP1}, respectively, then we could use the Galerkin method to compute
upper bounds $\lambda^*$ and $\liiu$, check if  $\lambda^* \leq \liil$
and the inequality $D_1 \leq D_2$ would then guarantee the validity of
\ttg{eq:closest}.
However, the guaranteed lower bounds $\lil$ and $\liil$ are not available.

Therefore, in practice we propose to set $\lil = X_2^2$, see \ttg{eq:absbound},
$\liil = (\lambda_* + \liiu)/2$, and compute $D_1$ and $D_2$ with these values.
This yields the following diagnostics indicating
the validity of the assumption \ttg{eq:closest}.
If $D_2 \leq 0$ or $D_1 > D_2$ then it is highly probable that some of the
assumptions is not satisfied and the results should not be trusted.
On the other hand if $D_2 > 0$ and $D_1$ is several times smaller than $D_2$
then we can have a good confidence in the validity of \ttg{eq:closest}.
This diagnostics performs very well in all numerical experiments presented below.
At the early stages of the computations the approximations are not very precise
and the diagnostics showed that the results are untrustworthy.
However, at the final stages of the computations
the value $D_1$ is at least five times smaller than $D_2$ providing good
confidence in the validity of \ttg{eq:closest}.

\section{Application to Friedrichs' inequality}
\label{se:Fried}

In this section we apply the above general theory to the case of Friedrichs'
inequality. We will consider the variant of Friedrichs' inequality that is
suitable for general symmetric second-order linear elliptic differential operators.
First, we introduce differential operators of this general type, see Subsection~\ref{se:elloper},
and provide the corresponding Friedrichs' inequality, see Subsection~\ref{se:fioc}.
In Subsection~\ref{se:Friedbounds}
we derive two-sided bounds on the optimal constant in Friedrichs'
inequality that are based on the general theory and on the complementarity
technique. In Subsection~\ref{se:Friedflux} certain computational issues are discussed and finally, in Subsection~\ref{se:Friedexper} we present numerical
results.

\subsection{A general symmetric second-order linear elliptic operator}
\label{se:elloper}
Let us consider a domain $\Omega \subset \R^d$ with Lipschitz boundary
$\partial\Omega$.
Let $\partial\Omega$ consist of two relatively open parts
$\GammaD$ and $\GammaN$ such that
$\partial\Omega = \oGammaD \cup \oGammaN$ and
$\GammaD \cap \GammaN = \emptyset$.
Note that we admit the case where either $\GammaD$ or $\GammaN$ is empty.
Further, let us consider a matrix function
$\cA \in [L^\infty(\Omega)]^{d\times d}$, coefficients
$c \in L^\infty(\Omega)$, and $\alpha \in L^\infty(\GammaN)$.
Matrix $\cA$ is assumed to be symmetric and uniformly positive definite,
i.e. there exists a constant $C>0$ such that
$$
  \b{\xi}^T \cA(x) \b{\xi} \geq C |\b{\xi}|^2 \quad\forall \b{\xi} \in \R^d
  \text{ and for a.a. } x\in\Omega,
$$
where $|\cdot|$ stands for the Euclidean norm.
Coefficients $c$ and $\alpha$ are considered to be nonnegative.

Further, we introduce a subspace
$$
  H^1_{\GammaD}(\Omega) = \{ v \in H^1(\Omega) : v = 0 \text{ on } \GammaD
  \text{ in the sense of traces} \}
$$
of the Sobolev space $H^1(\Omega)$ of $L^2(\Omega)$ functions with square
integrable distributional derivatives.
In what follows, we use the notation $(\cdot,\cdot)$ and $(\cdot,\cdot)_{L^2(\GammaN)}$
for the $L^2(\Omega)$ and $L^2(\GammaN)$ scalar products, respectively.
Using this notation, we define a bilinear form
\begin{equation}
\label{eq:blf}
  a(u,v) = (\cA \nabla u,\nabla v) + (c u, v) + (\alpha u, v)_{L^2(\GammaN)}.
\end{equation}

\subsection{Friedrichs' inequality and the optimal constant}
\label{se:fioc}
The bilinear form~\ttg{eq:blf} is a scalar product in $H^1_{\GammaD}(\Omega)$ under the conditions presented in the following lemma.
Its proof follows for instance from \cite[Theorem~5.11.2]{Kuf_John_Fucik_Function_spaces}.
\begin{lemma}
\label{le:scalprod}
The bilinear form defined in \ttg{eq:blf} is a scalar product
in $H^1_{\GammaD}(\Omega)$ provided that at least one of the following conditions
is satisfied:
\begin{itemize}
\item[$\mathrm{(a)}$] $\operatorname{meas}_{d-1} \GammaD > 0$,
\item[$\mathrm{(b)}$] there exists a nonempty ball $B \subset \Omega$ such that
$c > 0$ on $B$,
\item[$\mathrm{(c)}$] there exists a subset $\Gamma_\alpha \subset \GammaN$ such that
$\operatorname{meas}_{d-1} \Gamma_\alpha > 0$ and $\alpha > 0$ on $\Gamma_\alpha$.
\end{itemize}
\end{lemma}

In this section we assume that at least one of conditions (a)--(c) of
Lemma~\ref{le:scalprod} is satisfied, hence, that $a({\cdot},{\cdot})$ is a scalar product
in $H^1_{\GammaD}(\Omega)$.
The equivalence of the norm induced by $a(\cdot,\cdot)$ and the standard $H^1$ norm yields completeness and therefore the space $H^1_\GammaD(\Omega)$ equipped with the scalar product $a(\cdot,\cdot)$ is a Hilbert space. This enables to use the theory from Section~\ref{se:abseigenprob}. We set
\begin{equation}
\label{eq:Friedsetting}
 V = H^1_\GammaD(\Omega), \quad
 (u,v)_V = a(u,v), \quad
 H = L^2(\Omega), \quad
 (u,v)_H = (u,v),
\end{equation}
and we define $\gamma: H^1_\GammaD(\Omega) \rightarrow  L^2(\Omega)$
as the identity operator, which is compact due to the Rellich theorem
\cite[Theorem~6.3]{Adams_Sob_spaces_03}.
With this setting and with the notation
$\trinorm{\cdot}$ for the norm induced by $a(\cdot,\cdot)$,
we obtain the validity of Friedrichs' inequality.
\begin{theorem}
\label{th:Fried}
Let the bilinear form $a(\cdot,\cdot)$ given by \ttg{eq:blf} form a scalar product in $H^1_\GammaD(\Omega)$.
Then there exists a constant $\CF > 0$ such that
\begin{equation}
\label{eq:Fried}
  \norm{v}_{L^2(\Omega)} \leq \CF \trinorm{v} \quad \forall v \in H^1_\GammaD(\Omega).
\end{equation}
Moreover, the optimal value of this constant is $\CF = \lambda_1^{-1/2}$, where
$\lambda_1$ is the smallest eigenvalue of the following problem:
find $u_i \in H^1_\GammaD(\Omega)$, $u_i \neq 0$, and $\lambda_i \in\R$
such that
\begin{equation}
\label{eq:Friedeigenp}
  a(u_i,v) = \lambda_i ( u_i, v ) \quad \forall v \in H^1_\GammaD(\Omega).
\end{equation}
\end{theorem}
\begin{proof}
Lemma~\ref{le:scalprod} together with the equivalence of the norms $\trinorm{{\cdot}}$ and $\|{\cdot}\|_{H^1(\Omega)}$ guarantees that $H^1_\GammaD(\Omega)$ with the scalar product $a({\cdot},{\cdot})$
given by~\ttg{eq:blf} is a Hilbert space.
The statement then follows immediately from Theorem~\ref{th:absineq}.
\end{proof}

Let us note that the most common version of Friedrichs' inequality
$$
  \norm{v}_{L^2(\Omega)} \leq \CF \norm{\nabla v}_{L^2(\Omega)}
  \quad \forall v \in H^1_0(\Omega)
$$
follows from \ttg{eq:Fried} with $\GammaD = \partial\Omega$, $\GammaN = \emptyset$,
$\cA$ being identity matrix, $c = 0$, and $\alpha = 0$.
As usual, we denote by $H^1_0(\Omega)$ the space $H^1_\GammaD(\Omega)$ with
$\GammaD = \partial\Omega$.

\subsection{Two-sided bounds on Friedrichs' constant}
\label{se:Friedbounds}

A lower bound on Friedrichs' constant $\CF$ can be efficiently
computed by the Galerkin method. We use the setting \ttg{eq:Friedsetting}
and proceed as it is described in Section~\ref{se:absupper}.
The upper bound on $\CF$ is obtained by the complementarity technique
presented in the following theorem.

Let $\Hdiv$ stands for the space of $d$-dimensional vector fields with
square integrable weak divergences and let $\bn$ be the unit outward-facing normal vector
to the boundary $\partial\Omega$. Further, let
$\norm{\bq}_\cA^2 = (\cA \bq, \bq)$ be a norm in $[L^2(\Omega)]^d$
induced by the matrix $\cA$.

\begin{theorem}\label{th:comp}
Let $V = H^1_\GammaD(\Omega)$, $u_* \in V$ and $\lambda_* \in \R$.
Let the bilinear form $a(\cdot,\cdot)$ given by \ttg{eq:blf} form a scalar product in $V$.
Let $w\in V$ satisfy
\begin{equation}
\label{eq:defw}
  a(w,v) = a(u_*,v) - \lambda_*(u_*,v) \quad \forall v \in V,
\end{equation}
where recall that $(\cdot,\cdot)$ stands for the $L^2(\Omega)$ scalar product.
Then
\begin{equation}
  \label{eq:wbound}
  \trinorm{w} \leq \norm{\nabla u_* - \cA^{-1} \bq}_{\cA}
      + \CF \norm{\lambda_* u_* - c u_* + \ddiv \bq}_{L^2(\Omega)}
\quad \forall \bq \in W,
\end{equation}
where $W = \{ \bq \in \Hdiv : \bq\cdot \bn = -\alpha u_* \text{ on }\GammaN \}$.
\end{theorem}
\begin{proof}
Let us fix any $\bq \in W$, test \ttg{eq:defw} by $v=w$ and
use the divergence theorem to express
\begin{multline*}
  \trinorm{w}^2 = (\cA \nabla u_*,\nabla w) + (c u_*, w)
    + (\alpha u_*, w)_{L^2(\GammaN)} - \lambda_* (u_*,w)
    - (\bq,\nabla w) - (\ddiv\bq, w)
\\
+ (\bq\cdot\bn,w)_{L^2(\GammaN)}
  = \left(\cA (\nabla u_* - \cA^{-1} \bq),\nabla w\right)
  - (\lambda_* u_* - c u_* + \ddiv\bq,w).
\end{multline*}
The Cauchy--Schwarz inequality and Friedrichs' inequality \ttg{eq:Fried}
yield
$$
\trinorm{w}^2 \leq \norm{\nabla u_* - \cA^{-1} \bq}_\cA \norm{\nabla w}_\cA
  + \CF \norm{\lambda_* u_* - c u_* + \ddiv\bq}_{L^2(\Omega)} \trinorm{w}.
$$
Inequality $\norm{\nabla w}_\cA \leq \trinorm{w}$ finishes the proof.
\end{proof}

The estimate \ttg{eq:wbound} is of type \ttg{eq:abscompl} with
\begin{equation}
\label{eq:FriedAB}
A = \norm{\nabla u_* - \cA^{-1} \bq}_{\cA}, \quad
B = \norm{\lambda_* u_* - c u_* + \ddiv \bq}_{L^2(\Omega)},
\end{equation}
and $C_\gamma = \CF$.
The numbers $A$ and $B$ can be readily computed as soon as suitable approximations
$\lambda_*$, $u_*$, and a suitable vector field $\bq \in W$ are in hand.
Estimate \ttg{eq:absbound} then gives a guaranteed upper bound on
Friedrichs' constant $\CF$.

The crucial part is the computation of suitable approximations $\lambda_*$
and $u_*$ of the smallest eigenvalue $\lambda_1$ of problem
\ttg{eq:Friedeigenp} and its corresponding eigenvector $u_1$
such that the inequality \ttg{eq:closest} is satisfied.
Equally crucial is a suitable choice of the vector field $\bq \in W$
in such a way that estimate \ttg{eq:wbound} provides a tight upper bound
on $\trinorm{w}$.
A possible approach and practical details about these issues
are provided in the next section.

\subsection{Flux reconstruction}
\label{se:Friedflux}

In order to compute two-sided bounds on Friedrichs' constant, we proceed
as follows.
First, we use the Galerkin method, see Section~\ref{se:absupper}, to compute
an approximation $\lambda_1^h$ and $u_1^h$ of the eigenpair $\lambda_1$
and $u_1$ of~\ttg{eq:Friedeigenp}, respectively.
We set $\lambda_* = \lambda_1^h$, $u_* = u_1^h$, compute a suitable
vector field $\bq \in W$, and evaluate the numbers $A$ and $B$
by \ttg{eq:FriedAB}. The estimate \ttg{eq:absbound} then provides the upper bound and the Galerkin approximation $\lambda_1^h$ yields the lower bound \ttg{eq:abslowb} on Friedrichs' constant $\CF$.

The key point is the computation of the suitable vector field $\bq$.
For simplicity, we choose a straightforward approach of approximate
minimization of the upper bound \ttg{eq:wbound} with respect to a suitable
subset of $W$.
First, we exploit the affine structure of $W$.
Let us choose an arbitrary but fixed $\overline \bq \in W$.
The practical construction of this $\overline \bq$ is a geometrical issue
depending on $\Omega$ and coefficient $\alpha$.
It suffices to construct a vector field $\overline \bq_1 \in \Hdiv$
such that $\overline \bq_1 \cdot \bn = 1$ on $\GammaN$
and a function $\overline\alpha \in H^1(\Omega)$ such that $\overline\alpha = \alpha$
on $\GammaN$. Let us note that usually Raviart--Thomas--N\'{e}d\'{e}lec space is considered for construction of vector fields in $\Hdiv$.
Then we can simply set $\overline \bq = -\overline\alpha u_* \overline \bq_1$.
This $\overline \bq$ obviously satisfies the boundary condition
$\overline \bq \cdot \bn = -\alpha u_*$ on $\GammaN$.
To guarantee $\overline \bq \in \Hdiv$ we need an extra smoothness of
$\overline\bq_1$ and $\overline\alpha$. Since $u_* = u_1^h$ is piecewise polynomial and thus in $W^{1,\infty}(\Omega)$, it suffices to have $\overline\bq_1 \in [H^1(\Omega)]^d$
and $\overline\alpha \in W^{1,\infty}(\Omega)$, or in the opposite way
$\overline\bq_1 \in [W^{1,\infty}(\Omega)]^d$ and $\overline\alpha \in H^1(\Omega)$.


In any case, having a $\overline \bq \in W$, we can express the affine space $W$ as
$W = \overline \bq + W_0$, where
$$
  W_0 = \{ \bq \in \Hdiv : \bq\cdot \bn = 0 \text{ on }\GammaN \}
$$
is already a linear space.
The idea is to minimize the upper bound \ttg{eq:wbound} over the set
$\overline\bq +W_0^h$, where  $W_0^h \subset W_0$ is a finite dimensional
subspace. However, the right-hand side of \ttg{eq:wbound} is a nonlinear
functional in $\bq$ and, thus, in order to simplify the computation,
we use the idea from \cite{Rep:2008} and
approximate it by a quadratic functional.
We rewrite inequality \ttg{eq:wbound} using notation~\ttg{eq:FriedAB} and Young's inequality as
\begin{equation}
  \label{eq:AsqBsq}
  \trinorm{w}^2 \leq (A + \CF B)^2
\leq (1+\varrho^{-1}) A^2 + (1+\varrho) \CF^2 B^2
\quad \forall \varrho > 0.
\end{equation}
The right-hand side of this inequality is already a quadratic functional for
a fixed $\varrho$, but the exact value of $\CF$ is unknown in general.
However, it is sufficient to find an approximate minimizer only.
Therefore we approximate $\CF$ by the available value $(\lambda_1^h)^{-1/2}$.
This leads us to the minimization of
\begin{equation}
  \label{eq:Friedquadfunc}
(1+\varrho^{-1}) \norm{\nabla u_1^h - \cA^{-1} \bq}_{\cA}^2
  + (1+\varrho) (\lambda_1^h)^{-1}
    \norm{\lambda_1^h u_1^h - c u_1^h + \ddiv \bq}_{L^2(\Omega)}^2
\end{equation}
over the affine set $\overline\bq + W_0^h$ with a fixed value of $\varrho > 0$.
This minimization problem is equivalent to seeking $\bq^h_0 \in W_0^h$
satisfying
\begin{equation}
  \label{eq:divdivprob}
  \cB(\bq_0^h,\bw_0^h) = \cF(\bw_0^h) - \cB(\overline \bq, \bw_0^h)
\quad \forall \bw_0^h \in W_0^h,
\end{equation}
where
\begin{align*}
  \cB(\bq,\bw) &= (\ddiv \bq, \ddiv \bw) + \frac{\lambda_1^h}{\varrho}
     (\cA^{-1} \bq, \bw), \\
  \cF(\bw) &= \frac{\lambda_1^h}{\varrho} (\nabla u_1^h,\bw)
    - (\lambda_1^h u_1^h - c u_1^h, \ddiv \bw).
\end{align*}
The computed vector field $\bq^h = \overline\bq + \bq_0^h \in W$ is then
used in \ttg{eq:FriedAB} to evaluate $A$ and $B$ and consequently the
two-sided bounds $\CFlow \leq \CF \leq \CFup$, where
\begin{equation}
\label{eq:Friedtwosided}
\CFlow = (\lambda_1^h)^{-1/2}
\quad\text{and}\quad
\CFup = 1/X_2,
\end{equation}
$X_2$ is given by \ttg{eq:X2} with $\gamma u_* = u_1^h$,
see \ttg{eq:abslowb} and \ttg{eq:absbound}.

Note that problem \ttg{eq:divdivprob} can be naturally approached by
standard Raviart--Thomas--N\'{e}d\'{e}lec finite elements \cite{BreFor:1991}.
Further we note that a particular value of the constant $\varrho$
can influence the accuracy of the final bound. However, this influence
was minor in all cases we numerically tested and the natural value $\varrho = 1$
yielded accurate results. If necessary a simultaneous minimization of
\ttg{eq:Friedquadfunc} with respect to both $\varrho > 0$ and
$\bq \in \overline\bq + W_0^h$ can be performed.

\subsection{Numerical experiment}
\label{se:Friedexper}

In order to illustrate the capabilities of the above described approach
for computation of two-sided bounds on Friedrichs' constant, we present
numerical results showing the dependence of Friedrichs' constant
on piecewise constant values of $\cA$.
We consider the general setting from Section~\ref{se:elloper}
and in particular we set
$\Omega = (-1,1)^2$,
$\GammaN = \{ (x_1,x_2) \in \R^2 : x_1 = 1 \text{ and } -1 < x_2 < 1 \}$,
$\GammaD = \partial\Omega \setminus \oGammaN$,
$c=0$, and $\alpha=0$.
The matrix $\cA$ is piecewise constant, defined as
$\cA(x_1,x_2) = I$ for $x_1 x_2 \leq 0$ and
$\cA(x_1,x_2) = \tilde a I$ for $x_1 x_2 > 0$,
where
$I \in\R^{d\times d}$ stands for the identity matrix and
the value of the constant $\tilde a$ is specified below.

We employ the standard piecewise linear triangular finite elements
to discretize eigenvalue problem \ttg{eq:Friedeigenp}.
Thus, we consider a conforming triangular mesh $\cT_h$ and seek the Galerkin solution of problem \ttg{eq:Friedeigenp} in space
\begin{equation}
\label{eq:FriedVh}
  V^h = \{ v^h \in H^1_\GammaD(\Omega) : v^h|_K \in P^1(K),\ \forall K \in \cT_h \},
\end{equation}
where $P^1(K)$ denotes the space of affine functions on $K$,
see Section~\ref{se:absupper}.

We point out that the discontinuity of $\cA$
causes a strong singularity of the eigenvectors at the origin.
Therefore, we employ a standard adaptive algorithm, see Algorithm~\ref{al:1},
and construct adaptively refined meshes
in order to approximate the singularity well.
We use the localized version of \ttg{eq:Friedquadfunc} to define the error
indicators
\begin{equation}
\label{eq:FriedetaK}
 \eta_K^2 = (1+\varrho^{-1}) \norm{\nabla u_1^h - \cA^{-1} \bq}_{\cA,K}^2
  + (1+\varrho) (\lambda_1^h)^{-1}
    \norm{\lambda_1^h u_1^h - c u_1^h + \ddiv \bq}_{L^2(K)}^2
\end{equation}
for all $K\in\cT_h$, where $\norm{\bw}_{\cA,K}^2 = (\cA\bw,\bw)_{L^2(K)}$.

This error indicator is a natural choice, because
we compute both $u_1^h$
and an approximation of $\bq$ on the same mesh and hence we need an indicator that
combines errors of both these solutions.
Quantities \ttg{eq:FriedetaK} do it naturally.
Indeed, if $u_1^h$ is not accurate then its residual
$w$ is large and $\eta_K$ show high values, because the sum of their squares
approximates the upper bound \ttg{eq:AsqBsq}.
Similarly, if $\bq^h$ is far from the exact minimizer $\bq$ then $\eta_K$
exhibit high values as well.
In addition, quantities $\eta_K$ are readily available, because
both norms in \ttg{eq:FriedetaK} have to be computed anyway
to get the upper bound $\CFup$.

We note that all numerical experiments in this paper use own Matlab finite element
implementation. Meshes are refined by routines of the Matlab \texttt{PDEtoolbox} and
generalized eigenvalue problems for resulting sparse
matrices are solved by the Matlab routine \texttt{eigs} that is based
on the \texttt{ARPACK} package \cite{LehSorYan:1998}.

\begin{algorithm}
\centerline{\bfseries Adaptive algorithm}
\begin{description}
\itemsep=3pt
\item[Step 1] Construct an initial mesh $\cT_h$.

\item[Step 2] Use the space \ttg{eq:FriedVh}
and compute Galerkin approximations
$\lambda_1^h \in \R$ and $u_1^h \in V^h$
of the smallest eigenvalue and the corresponding eigenvector of problem
\ttg{eq:Friedeigenp}, see \ttg{eq:EP1h}.

\item[Step 3] Use the finite element space
$
  W^h_0 = \{ \bw_h \in W_0 : \bw_h \in [P^2(K)]^2,\ \forall K \in \cT_h \}
$
and solve \ttg{eq:divdivprob}.

\item[Step 4] \label{F_2sided_bound_FC} Evaluate two-sided bounds \ttg{eq:Friedtwosided}.
Set $\CFavg = (\CFup + \CFlow)/2$, $\RelErr = (\CFup - \CFlow) / \CFavg$,
and stop the algorithm as soon as $\RelErr \leq \ErrTol$.

\item[Step 5] Compute error indicators \ttg{eq:FriedetaK} for all $K\in\cT_h$
and sort them in descending order:
$\eta_{K_1} \geq \eta_{K_2} \geq \dots \geq \eta_{K_{\Nel}}$,
where $\Nel$ is the number of elements in $\cT_h$.

\item[Step 6] (Bulk criterion.) Find the smallest $n$ such that
$
  \theta^2 \sum_{i=1}^{\Nel} \eta_{K_i}^2 \leq \sum_{i=1}^{n} \eta_{K_i}^2,
$
where $\theta\in(0,1)$ is a parameter.

\item[Step 7] Construct a new mesh $\cT_h$ by refining elements $K_1$, $K_2$, \dots, $K_n$.

\item[Step 8] Go to Step 2.
\end{description}
\medskip

\caption{\label{al:1} Mesh adaptation algorithm for two-sided bounds on Friedrichs' constant.
}
\end{algorithm}

The unknown value of Friedrichs' constant lies between bounds
$\CFlow$ and $\CFup$ computed in Step~4 of Algorithm~\ref{al:1}.
When Algorithm~\ref{al:1} stops, the relative error is guaranteed to be at most the given tolerance $\ErrTol$.

In this particular numerical experiment we have chosen
the initial mesh with eight triangles as shown in Figure~\ref{fi:initmesh}.
The marking parameter in Step 6 and
the tolerance for the relative error in Step 4 were chosen
as $\theta=0.75$ and $\ErrTol = 0.01$, respectively.
Further, we naturally set $\varrho=1$.
We made an attempt to find an optimal value for $\varrho$.
However, this decreases the relative error by a factor of magnitude
$10^{-4}$--$10^{-3}$ on a fixed mesh, which is below the target accuracy
$\ErrTol = 0.01$. Therefore, we use the natural value $\varrho=1$ that equilibrates
both terms in \ttg{eq:FriedetaK}. A similar statement can be made in 
all subsequent numerical experiments. On the other hand, certain
examples might behave differently and optimization of $\varrho$ could
be important.

\begin{figure}
\begin{center}
\vspace{6pt}
\makebox[0pt][l]{\hspace{-24pt}\raisebox{-12pt}[0pt][0pt]{$(-1,-1)$}}%
\makebox[0pt][l]{\hspace{35mm}\raisebox{-12pt}[0pt][0pt]{$(1,-1)$}}%
\makebox[0pt][l]{\hspace{36mm}\raisebox{42mm}[0pt][0pt]{$(1,1)$}}%
\makebox[0pt][l]{\hspace{-24pt}\raisebox{42mm}[0pt][0pt]{$(-1,1)$}}%
\makebox[0pt][l]{\hspace{-16pt}\raisebox{19mm}[0pt][0pt]{$\GammaD$}}%
\makebox[0pt][l]{\hspace{41.5mm}\raisebox{19mm}[0pt][0pt]{$\GammaN$}}%
\makebox[0pt][l]{\hspace{19mm}\raisebox{42mm}[0pt][0pt]{$\GammaD$}}%
\makebox[0pt][l]{\hspace{19mm}\raisebox{-12pt}[0pt][0pt]{$\GammaD$}}%
\includegraphics{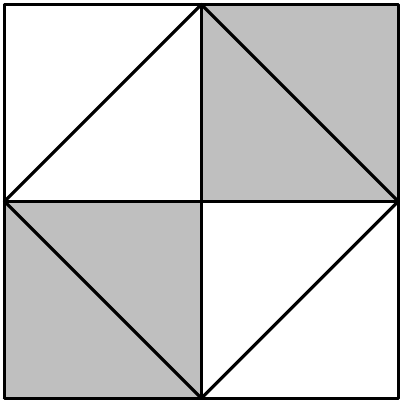}
\\[12pt]
\caption{
The initial mesh for the adaptive algorithm.
The Dirichlet and Neumann parts of the boundary are indicated as well as
the piecewise constant matrix $\cA$:
$\cA=I$ in white elements and $\cA=\tilde a I$ in gray elements.}
\label{fi:initmesh}
\end{center}
\end{figure}

The results for a series of values of $\tilde a$ 
are summarized in Table~\ref{ta:Friedres}.
For each particular value of $\tilde a$, we run the adaptive algorithm
until the relative error drops below $\ErrTol = 0.01$.
This error level was reached in all cases using several thousands
of degrees of freedom.
Notice the considerable dependence of the optimal value of
Friedrichs' constant on $\tilde a$.
The values for $\tilde a=0.001$ and $\tilde a=1000$ differ
more than thirty times.
Further notice that the exact value of Friedrichs' constant
for $\tilde a=1$ is $\CF=4/(\pi\sqrt{5})\approx 0.5694$. 

In order to illustrate the behavior of the adaptive process, we consider
the case $\tilde a=0.001$. The convergence of the bounds $\CFup$ and $\CFlow$
is presented in Figure~\ref{fi:resCF} (left).
The right panel of Figure~\ref{fi:resCF} shows the convergence
of the relative error $\RelErr$ together with the heuristic indicators
$D_1$ and $D_2$ introduced at the end of Section~\ref{se:lbound}.
The fact that $D_1$ is several times smaller than $D_2$ at the later stages
of the adaptive process indicates good confidence
in the validity of assumption \ttg{eq:closest}.
In addition, two adapted meshes are drawn in
Figures~\ref{fi:meshesA}--\ref{fi:meshesB} (left).

\begin{table}
\center{
\begin{tabular}{r|cccr}
$\tilde a$ & $\CFlow$ & $\CFup$ & $\RelErr$ & $\NDOF$
\\ \hline
$0.001$ & 9.0086 & 9.0939 & 0.0094 & 4\,832 \\
$0.01$  & 2.8697 & 2.8971 & 0.0095 & 5\,003 \\
$0.1$   & 1.0035 & 1.0124 & 0.0088 & 7\,866 \\
$1$     & 0.5693 & 0.5743 & 0.0086 & 4\,802 \\
$10$    & 0.3173 & 0.3201 & 0.0088 & 7\,866 \\
$100$   & 0.2870 & 0.2897 & 0.0095 & 5\,003 \\
$1000$  & 0.2849 & 0.2876 & 0.0094 & 4\,832 \\
\end{tabular}
}
\caption{
Friedrichs' constant.
The lower bound $\CFlow$, upper bound $\CFup$, relative error $\RelErr$,
and the number of degrees of freedom $\NDOF=\operatorname{dim}V^h$ for
particular values of $\tilde a$.}
\label{ta:Friedres}
\end{table}

\begin{figure}
\begin{center}
\includegraphics[height=50mm]{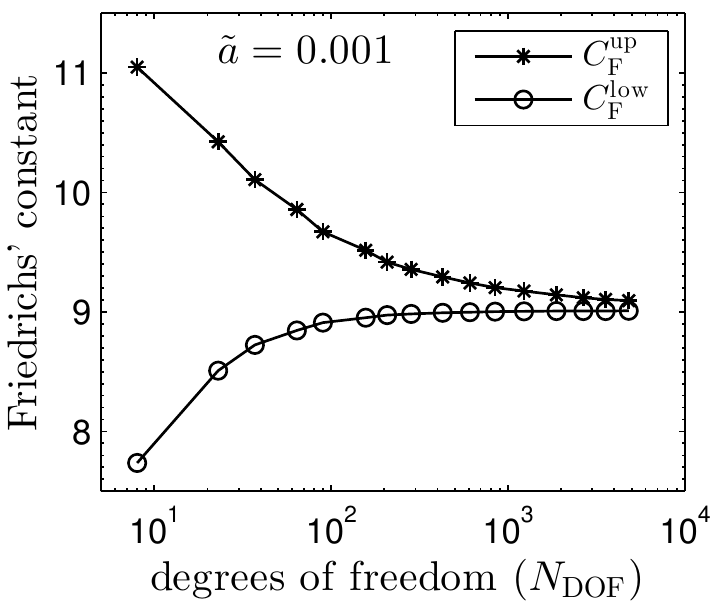}\quad
\includegraphics[height=50mm]{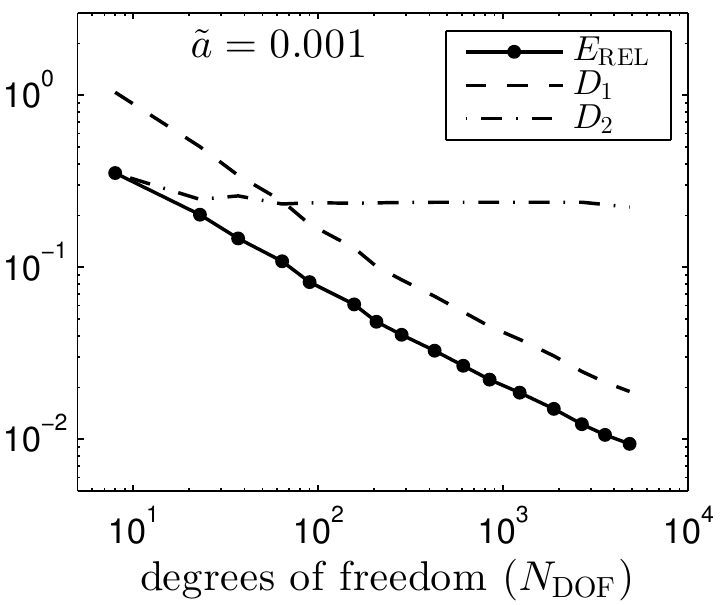}
\caption{
Friedrichs' constant for $\tilde a = 0.001$.
Convergence of bounds $\CFup$ and $\CFlow$ (left)
and of the relative error $\RelErr$ and heuristic indicators $D_1$ and $D_2$ (right).
}
\label{fi:resCF}
\end{center}
\end{figure}

\section{Application to the Poincar\'e inequality}

\subsection{Poincar\'e inequality and the optimal constant}
In this section we consider the case when none of conditions (a)--(c) of
Lemma~\ref{le:scalprod} is satisfied. Therefore, we assume
the general symmetric
second-order elliptic operator as described in Section~\ref{se:elloper}
with $c=0$, $\alpha=0$, $\GammaD=\emptyset$, and
$\GammaN=\partial\Omega$.
We apply the general theory of Section~\ref{se:absineq} with
$V = \oHI = \{ v \in H^1(\Omega) : (v,1) = 0 \}$,
$(u,v)_V = a(u,v) = (\cA \nabla u, \nabla v)$,
$H = L^2(\Omega)$, and
$(u,v)_H = (u,v)$.
It is an easy exercise to verify that $a(\cdot,\cdot)$ forms a scalar product
on $\oHI$ and that it induces a norm $\trinorm{\cdot}$
equivalent to the standard $H^1$-seminorm and $H^1$-norm.
The operator $\gamma: V \rightarrow H$ is set to be the identity mapping
$I : \oHI\rightarrow L^2(\Omega)$.
This mapping is clearly compact, because the identity mapping
from $\oHI$ to $H^1(\Omega)$ is linear and continuous
and the identity mapping from $H^1(\Omega)$ to $L^2(\Omega)$ is compact
due to the  Rellich theorem \cite[Theorem~6.3]{Adams_Sob_spaces_03}.
This setting enables to use the general conclusions of Theorem~\ref{th:absineq}
and obtain the following result.
\begin{theorem}
\label{th:Poin}
There exists a constant $\CP > 0$ such that
\begin{equation}
\label{eq:Poincare}
  \norm{v}_{L^2(\Omega)} \leq \CP \trinorm{v}
     \quad \forall v \in \oHI.
\end{equation}
Moreover, the optimal value of this constant is $\CP = \lambda_1^{-1/2}$, where
$\lambda_1$ is the smallest \emph{positive} eigenvalue of the following problem:
find $u_i \in H^1(\Omega)$, $u_i \neq 0$, and $\lambda_i \in\R$
such that
\begin{equation}
\label{eq:Poineigenp}
  a(u_i,v) = \lambda_i ( u_i, v ) \quad \forall v \in H^1(\Omega).
\end{equation}
\end{theorem}
\begin{proof}
The inequality \ttg{eq:Poincare} follows immediately from
Theorem~\ref{th:absineq}.
This theorem also implies that the optimal constant is $\CP = \tilde\lambda_1^{-1/2}$,
where $\tilde\lambda_1$ is the smallest eigenvalue of the following problem:
find $\tilde u_i \in \oHI$, $\tilde u_i \neq 0$, and $\tilde\lambda_i \in\R$
such that
\begin{equation}
\label{eq:Poineigenp2}
  a(\tilde u_i, v) = \tilde\lambda_i (\tilde u_i, v)
  \quad \forall v \in \oHI.
\end{equation}
Notice that $0 < \tilde\lambda_1 \leq \tilde\lambda_2 \leq \dots$.
Similarly, the eigenvalues of \ttg{eq:Poineigenp} satisfy
$0 = \lambda_0 < \lambda_1 \leq \lambda_2 \leq \dots$
and the zero eigenvalue corresponds to a constant eigenvector $u_0=1$.
It can be easily shown that $\tilde\lambda_i=\lambda_i$
and $\tilde u_i = u_i$ for all $i=1,2,\dots$.
Thus,
the smallest eigenvalue $\tilde\lambda_1$ of \ttg{eq:Poineigenp2}
is equal to the smallest \emph{positive} eigenvalue $\lambda_1$ of \ttg{eq:Poineigenp}
and the proof is finished.
\end{proof}

Let us note that the Poincar\'e inequality \ttg{eq:Poincare} can be equivalently
formulated as
\begin{equation}
\label{eq:Poincare2}
\norm{v - \tilde v}_{L^2(\Omega)} \leq \CP \trinorm{v}
     \quad \forall v \in H^1(\Omega),
\quad \tilde v = (v,1) / |\Omega|.
\end{equation}
The common version of the Poincar\'e inequality
$$
  \norm{v - \tilde v}_{L^2(\Omega)} \leq \CP \norm{\nabla v}_{L^2(\Omega)}
  \quad \forall v \in H^1(\Omega),\quad \tilde v = (v,1) / |\Omega|
$$
then follows from \ttg{eq:Poincare2}, or equivalently from~\ttg{eq:Poincare},
with $\cA$ being the identity matrix.

\subsection{Two-sided bounds on the Poincar\'e constant}

A lower bound on the Poincar\'e constant $\CP$ can be computed in
the standard way by the Galerkin method, see Section~\ref{se:absupper}.
The only difference is that here we compute the approximation $\lambda_1^h$ of
the second smallest (the smallest positive) eigenvalue of \ttg{eq:Poineigenp},
because the smallest eigenvalue is $\lambda_0 = 0$.
In order to compute the upper bound, we employ the complementarity technique
as follows.

\begin{theorem}\label{th:Poincomp}
Let $V = \oHI$, $u_* \in V$, and $\lambda_* \in \R$.
Let $w\in V$ satisfy
\begin{equation}
\label{eq:Poindefw}
  a(w,v) = a(u_*,v) - \lambda_*(u_*,v) \quad \forall v \in V.
\end{equation}
Then
\begin{equation}
  \label{eq:Poinwbound}
  \trinorm{w} \leq \norm{\nabla u_* - \cA^{-1} \bq}_{\cA}
      + \CP \norm{\lambda_* u_* + \ddiv \bq}_{L^2(\Omega)}
\quad \forall \bq \in W_0,
\end{equation}
where $W_0 = \{ \bq \in \Hdiv : \bq\cdot \bn = 0 \text{ on } \partial\Omega \}$.
\end{theorem}
\begin{proof}
By fixing arbitrary $\bq \in W_0$, testing \ttg{eq:Poindefw} by $v=w$ and
using the divergence theorem, we obtain
\begin{multline*}
  \trinorm{w}^2 = (\cA \nabla u_*,\nabla w) - \lambda_* (u_*,w)
    - (\bq,\nabla w) - (\ddiv\bq, w)
\\
  = \left(\cA (\nabla u_* - \cA^{-1} \bq),\nabla w\right)
  - (\lambda_* u_* + \ddiv\bq,w).
\end{multline*}
The Cauchy--Schwarz inequality and Poincar\'e inequality \ttg{eq:Poincare}
yield
$$
\trinorm{w}^2 \leq \norm{\nabla u_* - \cA^{-1} \bq}_\cA \norm{\nabla w}_\cA
  + \CP \norm{\lambda_* u_* + \ddiv\bq}_{L^2(\Omega)} \trinorm{w}.
$$
Since $\norm{\nabla w}_\cA = \trinorm{w}$, the proof is finished.
\end{proof}

We observe that complementarity estimate \ttg{eq:Poinwbound} is of type
\ttg{eq:abscompl} with
\begin{equation}
\label{eq:PoinAB}
A = \norm{\nabla u_* - \cA^{-1} \bq}_{\cA}, \quad
B = \norm{\lambda_* u_* + \ddiv \bq}_{L^2(\Omega)},
\end{equation}
and $C_\gamma = \CP$.
As soon as suitable approximations $\lambda_*$, $u_*$, and a suitable vector field $\bq \in W_0$ are available,
the numbers $A$ and $B$ can be computed and used in
\ttg{eq:absbound}--\ttg{eq:X2}
to obtain a guaranteed upper bound on the Poincar\'e constant $\CP$.

A straightforward approach for computation of $\lambda_*$, $u_*$, and $\bq$
was described in Section~\ref{se:Friedflux} for the case of Friedrichs' constant.
It can be directly used also for the Poincar\'e constant. It is even simpler,
because $c=0$, $\alpha=0$, and $\GammaN = \partial\Omega$.
The only difference is that approximations $\lambda_1^h$ and $u_1^h$ of
the second smallest (the first positive) eigenvalue of \ttg{eq:Poineigenp} and its corresponding eigenvector, respectively, have to be used.
In particular, the approximation $\bq^h \in W_0$ is computed using
 \ttg{eq:divdivprob}.
This vector field is then used in \ttg{eq:PoinAB} to evaluate $A$ and $B$
and consequently the two-sided bounds
$\CPlow \leq \CP \leq \CPup$, where
\begin{equation}
\label{eq:Pointwosided}
  \CPlow = (\lambda_1^h)^{-1/2}
\quad\text{and}\quad
  \CPup = 1/X_2,
\end{equation}
$X_2$ is given by \ttg{eq:X2} with $\gamma u_* = u_1^h$,
see also \ttg{eq:abslowb} and \ttg{eq:absbound}.

\subsection{Numerical experiment}

We consider the same setting as in Section~\ref{se:Friedexper}.
The only difference is that in the case of Poincar\'e constant we assume
$\GammaN=\partial\Omega$ and $\GammaD=\emptyset$.
We employ the adaptive algorithm as before (Algorithm~\ref{al:1})
with clear modifications. We use error indicators \ttg{eq:FriedetaK},
where $\lambda_1^h$ and $u_1^h$  are
the Galerkin approximations of the second smallest (the smallest positive)
eigenvalue $\lambda_1$ of \ttg{eq:Poineigenp} and its corresponding
eigenvector $u_1$, respectively. The relative error in Step~4 of Algorithm~\ref{al:1}
is computed using two-sided bounds \ttg{eq:Pointwosided}.

The obtained two-sided bounds on the Poincar\'e constant $\CP$ for a series of values of $\tilde a$ are
presented in Table~\ref{ta:Poinres}.
The guaranteed 0.01 relative error tolerance was reached in all cases
using several thousands degrees of freedom.
As in the case of Friedrichs' constant,
we observe considerable dependence of the Poincar\'e constant $\CP$ on $\tilde a$.
Finally, we point out that the exact value of the Poincar\'e constant
for $\tilde a=1$ is $\CP=2/\pi \approx 0.6366$. 

The progress of the adaptive algorithm is illustrated for the case
$\tilde a = 0.001$ in Figure~\ref{fi:resCP} (left), where
the convergence of bounds $\CPup$ and $\CPlow$ is shown.
Figure~\ref{fi:resCP} (right) presents the relative error $\RelErr$ and heuristic indicators
$D_1$, $D_2$, see the end of Section~\ref{se:lbound}.
In later stages of the adaptive process the indicator $D_1$ is considerably
smaller than $D_2$ which indicates good confidence
in the validity of assumption \ttg{eq:closest}.
Two adapted meshes are drawn in Figures~\ref{fi:meshesA}--\ref{fi:meshesB} (middle).

\begin{table}
\center{
\begin{tabular}{r|cccr}
$\tilde a$ & $\CPlow$ & $\CPup$ & $\RelErr$ & $\NDOF$
\\ \hline
$0.001$ & 14.2390 & 14.3690 & 0.0091 & 3\,400 \\
$0.01$  &  4.5199 &  4.5623 & 0.0093 & 3\,510 \\
$0.1$   &  1.4849 &  1.4989 & 0.0094 & 4\,382 \\
$1$     &  0.6365 &  0.6424 & 0.0092 & 3\,009 \\
$10$    &  0.4696 &  0.4740 & 0.0094 & 4\,382 \\
$100$   &  0.4520 &  0.4562 & 0.0093 & 3\,510 \\
$1000$  &  0.4503 &  0.4544 & 0.0091 & 3\,400 \\
\end{tabular}
}
\caption{
Poincar\'e constant.
The lower bound $\CPlow$, upper bound $\CPup$, relative error $\RelErr$,
and the number of degrees of freedom $\NDOF=\operatorname{dim}V^h$ for
particular values of $\tilde a$.}
\label{ta:Poinres}
\end{table}

\begin{figure}
\begin{center}
\includegraphics[height=50mm]{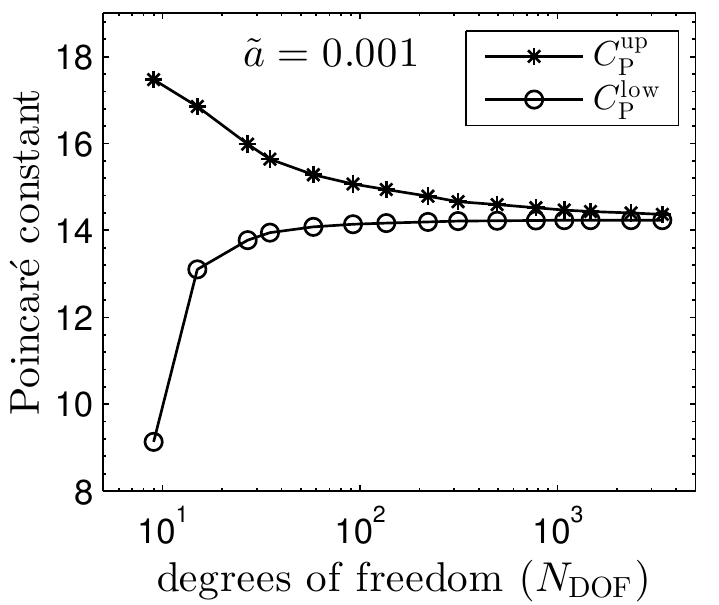}\quad
\includegraphics[height=50mm]{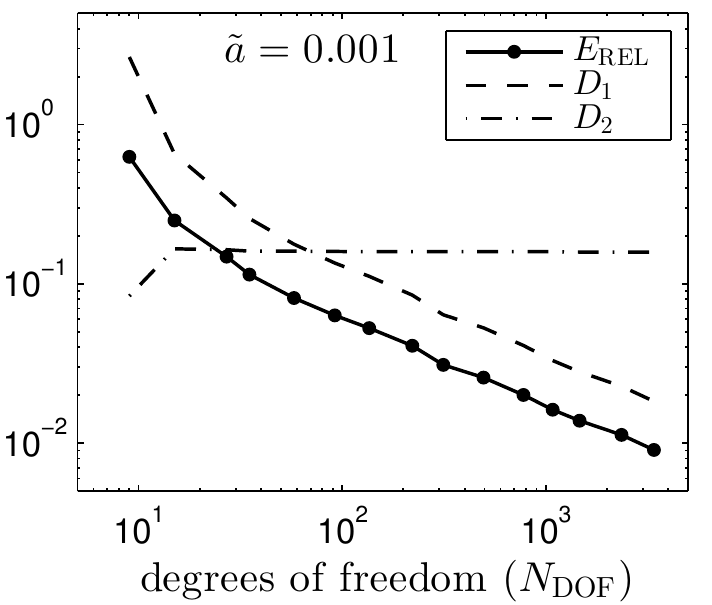}
\caption{
Poincar\'e constant for $\tilde a = 0.001$.
Convergence of bounds $\CPup$ and $\CPlow$ (left)
and of the relative error $\RelErr$ and heuristic indicators $D_1$ and $D_2$ (right).
}
\label{fi:resCP}
\end{center}
\end{figure}

\section{Application to the trace inequality}
\label{se:trace}

\subsection{Trace inequality and the optimal constant}
In order to apply the general theory from
Sections~\ref{se:abseigenprob}--\ref{se:absineq}
to the case of the trace inequality,
we consider the same general symmetric
second-order elliptic operator as in Section~\ref{se:elloper}.
In addition we assume $\operatorname{meas}_{d-1} \GammaN > 0$ and that at least one of conditions (a)--(c) of
Lemma~\ref{le:scalprod} is satisfied.
The general theory is applied with
$V = H^1_\GammaD(\Omega)$, $(u,v)_V = a(u,v)$,
$H = L^2(\GammaN)$, and $(u,v)_H = (u,v)_{L^2(\GammaN)}$.
The operator $\gamma: V \rightarrow H$ is the standard trace operator.
Its compactness and other properties are provided in \cite[Theorem 6.10.5]{Kuf_John_Fucik_Function_spaces}, see also \cite{Biegert:2009}.
The general result from Theorem~\ref{th:absineq} then translates as follows.
\begin{theorem}
\label{th:trace}
Let the bilinear form $a(\cdot,\cdot)$ given by \ttg{eq:blf} form a scalar product in $H^1_\GammaD(\Omega)$.
Then there exists a constant $\CT > 0$ such that
\begin{equation}
\label{eq:trace}
  \norm{v}_{L^2(\GammaN)} \leq \CT \trinorm{v} \quad \forall v \in H^1_\GammaD(\Omega).
\end{equation}
Moreover, the optimal value of this constant is $\CT = \lambda_1^{-1/2}$, where
$\lambda_1$ is the smallest eigenvalue of the following problem:
find $u_i \in H^1_\GammaD(\Omega)$, $u_i \neq 0$, and $\lambda_i \in\R$
such that
\begin{equation}
\label{eq:traceeigenp}
  a(u_i,v) = \lambda_i ( u_i, v )_{L^2(\GammaN)} \quad \forall v \in H^1_\GammaD(\Omega).
\end{equation}
\end{theorem}
\begin{proof}
The statement follows immediately from Lemma~\ref{le:scalprod} and
Theorem~\ref{th:absineq}.
\end{proof}
The common version of the trace inequality
$$
  \norm{v}_{L^2(\partial\Omega)} \leq \CT \norm{v}_{H^1(\Omega)}
  \quad \forall v \in H^1(\Omega)
$$
then follows from \ttg{eq:trace} with $\GammaD = \emptyset$,
$\GammaN = \partial\Omega$, $\cA$ being the identity matrix, $c = 1$,
and $\alpha = 0$.

\subsection{Two-sided bounds on the trace constant}

As in the case of Friedrichs' inequality, we compute the lower bound
of the optimal value for the constant $\CT$ by the Galerkin method,
see Section~\ref{se:absupper}.
In order to compute an upper bound on $\CT$ we employ the complementarity
technique as follows.

\begin{theorem}\label{th:tracecomp}
Let $V = H^1_\GammaD(\Omega)$, $u_* \in V$, and $\lambda_* \in \R$.
Let the bilinear form $a(\cdot,\cdot)$ given by \ttg{eq:blf} form a scalar product in $V$.
Let $w\in V$ satisfy
\begin{equation}
\label{eq:tracedefw}
  a(w,v) = a(u_*,v) - \lambda_*(u_*,v)_{L^2(\GammaN)} \quad \forall v \in V.
\end{equation}
Then, for any $\bq \in \Hdiv$
\begin{equation}
  \label{eq:tracewbound}
  \trinorm{w} \leq
            \norm{\nabla u_* - \cA^{-1} \bq}_{\cA}
      + \CF \norm{c u_* - \ddiv \bq}_{L^2(\Omega)}
      + \CT \norm{\alpha u_* - \lambda_* u_* + \bq \cdot \bn}_{L^2(\GammaN)}.
\end{equation}
\end{theorem}
\begin{proof}
Let us fix any $\bq \in \Hdiv$, test \ttg{eq:tracedefw} by $v=w$ and
use the divergence theorem to express
\begin{multline*}
  \trinorm{w}^2 = (\cA \nabla u_*,\nabla w) + (c u_*, w)
    + (\alpha u_*, w)_{L^2(\GammaN)} - \lambda_* (u_*,w)_{L^2(\GammaN)}
\\
    - (\bq,\nabla w) - (\ddiv\bq, w)
+ (\bq\cdot\bn,w)_{L^2(\GammaN)}
\\
  = \left(\cA (\nabla u_* - \cA^{-1} \bq),\nabla w\right)
  + (c u_* - \ddiv\bq,w)
  + (\alpha u_* - \lambda_* u_* + \bq\cdot\bn, w)_{L^2(\GammaN)}.
\end{multline*}
The Cauchy--Schwarz inequality, Friedrichs' inequality \ttg{eq:Fried},
and trace inequality \ttg{eq:trace}
yield
\begin{multline*}
\trinorm{w}^2 \leq \norm{\nabla u_* - \cA^{-1} \bq}_\cA \norm{\nabla w}_\cA
  + \CF \norm{c u_* - \ddiv\bq}_0 \trinorm{w}
\\
  + \CT \norm{\alpha u_* - \lambda_* u_* + \bq\cdot\bn}_{L^2(\GammaN)} \trinorm{w}.
\end{multline*}
The inequality $\norm{\nabla w}_\cA \leq \trinorm{w}$ finishes the proof.
\end{proof}

As in the case of Friedrichs' inequality, the bound \ttg{eq:tracewbound}
is of the type \ttg{eq:abscompl} with
\begin{equation}
\label{eq:traceAB}
 A = \norm{\nabla u_* - \cA^{-1} \bq}_{\cA}
      + \CF \norm{c u_* - \ddiv \bq}_{L^2(\Omega)},
\quad
  B = \norm{\alpha u_* - \lambda_* u_* + \bq \cdot \bn}_{L^2(\GammaN)}
\end{equation}
and $C_\gamma = \CT$.
Let us note that the complementarity estimate \ttg{eq:tracewbound}
is just one out of several possibilities. This bound comes
from \cite{Rep:2008} and contains Friedrichs' constant $\CF$.
Instead of its exact value an upper bound as computed in Section~\ref{se:Fried}
can be used here.
However, there exist other variants of
the complementarity technique that can be used to obtain a bound on
$\trinorm{w}$ avoiding the need of Friedrichs' constant $\CF$.
See for example \cite{robustaee:2010,BraSch:2008,LocalHypercycle:2006,Vohralik:2011}.

A suitable vector field $\bq \in \Hdiv$ is computed
by approximate minimization of the right-hand side of \ttg{eq:tracewbound}.
In a similar way as we obtained the functional \ttg{eq:Friedquadfunc},
we obtain the quadratic functional
\begin{multline}
  \label{eq:tracequadfunc}
  (1+\varrho^{-1})(1+\sigma^{-1}) \norm{\nabla u_1^h - \cA^{-1} \bq}_{\cA}^2
 +(1+\varrho)(1+\sigma^{-1})
     (\CFup)^2 \norm{c u_1^h - \ddiv \bq}_{L^2(\Omega)}^2
\\
 +(1+\sigma)
     (\lambda_1^h)^{-1} \norm{\alpha u_1^h - \lambda_1^h u_1^h + \bq \cdot \bn}_{L^2(\GammaN)}^2
\quad\forall \varrho > 0,\ \sigma > 0,
\end{multline}
where $\CFup$ is an upper bound of $\CF$ computed as described in
Sections~\ref{se:Friedbounds}--\ref{se:Friedflux},
$\lambda_1^h$ is the approximation of the smallest eigenvalue
of \ttg{eq:traceeigenp} obtained by the Galerkin method and $u_1^h$ is the
corresponding approximate eigenvector.
We look for the minimum of \ttg{eq:tracequadfunc} over a finite dimensional subspace
$W^h$ of $\Hdiv$. This minimization problem is equivalent to
seeking $\bq^h \in W^h$ such that
$$
  \cB(\bq^h,\bw) = \cF(\bw) \quad\forall \bw \in W^h,
$$
where
\begin{align*}
\cB(\bq,\bw) &=  \frac{1+\varrho}{\sigma} \lambda_1^h (\CFup)^2 (\ddiv\bq,\ddiv\bw)
  + \frac{1+\varrho}{\varrho\sigma} \lambda_1^h (\cA^{-1}\bq,\bw)
  + (\bq\cdot\bn, \bw \cdot \bn)_{L^2(\GammaN)},
\\
\cF(\bw) &= \frac{1+\varrho}{\varrho\sigma} \lambda_1^h (\nabla u_1^h, \bw)
  + \frac{1+\varrho}{\sigma} \lambda_1^h (\CFup)^2 (c u_1^h, \ddiv \bw)
  + (\lambda_1^h u_1^h - \alpha u_1^h, \bw \cdot \bn)_{L^2(\GammaN)}.
\end{align*}
Practically, the classical Raviart--Thomas--N\'{e}d\'{e}lec finite element method \cite{BreFor:1991}
can be used to solve this problem.
Note that the natural values for $\varrho$ and $\sigma$ are
$\varrho=1$ and $\sigma=2$, because then
$(1+\varrho^{-1})(1+\sigma^{-1}) = (1+\varrho)(1+\sigma^{-1}) = (1+\sigma) = 3$,
see \ttg{eq:tracequadfunc}.
Similarly to the case of Friedrichs' constant,
these natural values often yield accurate results.
If not, a simultaneous minimization of \ttg{eq:tracequadfunc}
with respect to $\varrho > 0$, $\sigma > 0$, and $\bq \in W^h$
can be performed.

The computed vector field $\bq^h \in W^h$ together with $\lambda_* = \lambda_1^h$ and $u_* = u_1^h$
are then substituted to \ttg{eq:traceAB}
in order to evaluate $A$ and $B$. Consequently, we obtain the
two-sided bounds $\CTlow \leq \CT \leq \CTup$, where
\begin{equation}
\label{eq:tracetwosided}
\CTlow = (\lambda_1^h)^{-1/2}
\quad\text{and}\quad
\CTup = 1/X_2,
\end{equation}
$X_2$ is given by \ttg{eq:X2} with $\gamma u_* = \gamma u_1^h$,
see \ttg{eq:abslowb} and \ttg{eq:absbound}.

\subsection{Numerical experiment}
\label{se:traceexper}

In order to illustrate the numerical performance of the above described method,
we consider the same example as in Section~\ref{se:Friedexper}.
We proceed in the same way as in Section~\ref{se:Friedexper}
with clear modifications in order to
compute two-sided bounds \ttg{eq:tracetwosided}
on the trace constant.
As before, the adaptive algorithm is steered by error indicators that
are in this case defined by a localized version of \ttg{eq:tracequadfunc}:
\begin{multline*}
  \eta_K^2 =
(1+\varrho^{-1})(1+\sigma^{-1}) \norm{\nabla u_1^h - \cA^{-1} \bq}_{\cA,K}^2
 +(1+\varrho)(1+\sigma^{-1})
     (\CFup)^2 \norm{c u_1^h - \ddiv \bq}_{L^2(K)}^2
\\
 +(1+\sigma)
     (\lambda_1^h)^{-1} \norm{\alpha u_1^h - \lambda_1^h u_1^h + \bq \cdot \bn}_{L^2(\partial K \cap \GammaN)}^2
\quad\forall K\in\cT_h.
\end{multline*}
The parameters $\varrho$ and $\sigma$ are
naturally chosen as $\varrho=1$ and $\sigma=2$.
The values of the upper bound $\CFup$ are taken from Table~\ref{ta:Friedres}.

\begin{table}
\center{
\begin{tabular}{r|cccr}
$\tilde a$ & $\CTlow$ & $\CTup$ & $\RelErr$ & $\NDOF$
\\ \hline
$0.001$ & 17.8110 & 17.9760 & 0.0092  & 5\,523 \\
$0.01$  &  5.6490 &  5.7047 & 0.0098  & 5\,418 \\
$0.1$   &  1.8433 &  1.8593 & 0.0086  & 7\,775 \\
$1$     &  0.7963 &  0.8033 & 0.0088  & 5\,499 \\
$10$    &  0.5829 &  0.5880 & 0.0086  & 7\,775 \\
$100$   &  0.5649 &  0.5705 & 0.0098  & 5\,421 \\
$1000$  &  0.5632 &  0.5685 & 0.0092  & 5\,523 \\
\end{tabular}
}
\caption{
Trace constant.
The lower bound $\CTlow$, upper bound $\CTup$, relative error $\RelErr$,
and the number of degrees of freedom $\NDOF=\operatorname{dim}V^h$ for
particular values of $\tilde a$.}
\label{ta:traceres}
\end{table}

\begin{figure}
\begin{center}
\includegraphics[height=50mm]{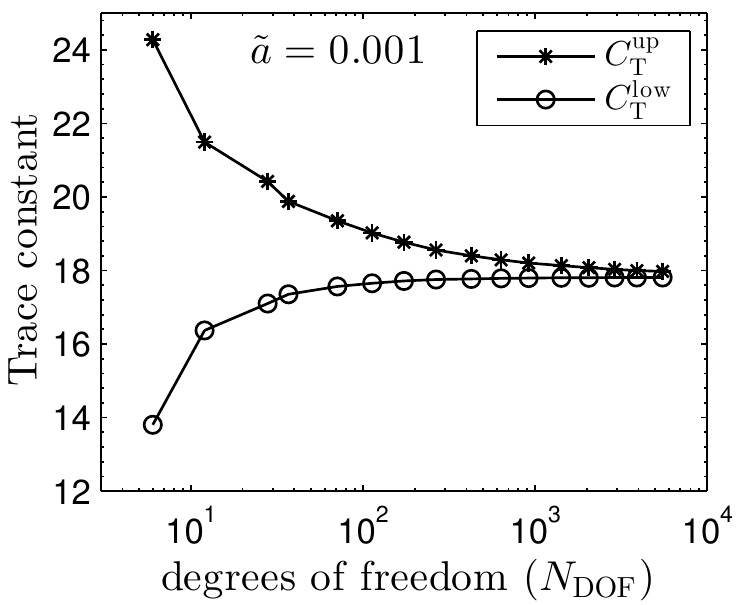}\quad
\includegraphics[height=50mm]{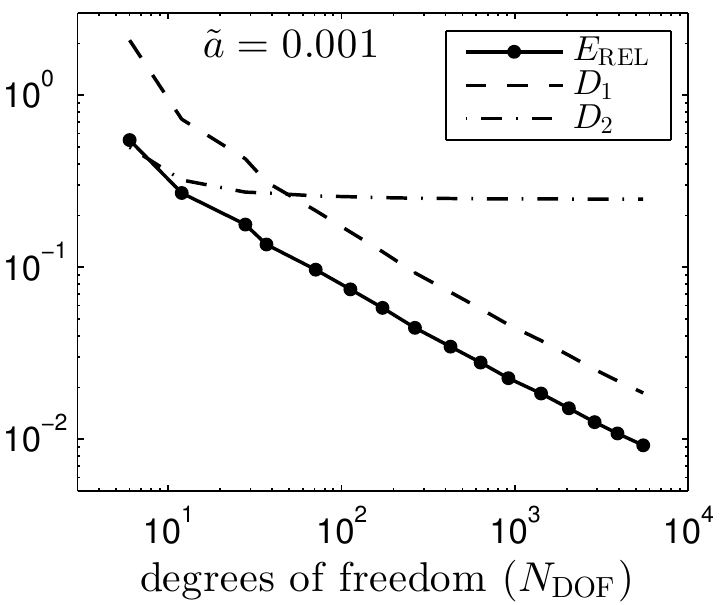}
\caption{
Trace constant for $\tilde a = 0.001$.
Convergence of bounds $\CTup$ and $\CTlow$ (left)
and of the relative error $\RelErr$ and heuristic indicators $D_1$ and $D_2$ (right).
}
\label{fi:resCT}
\end{center}
\end{figure}

The obtained results are presented in Table~\ref{ta:traceres}.
The method succeeded in obtaining guaranteed two-sided bounds
on the trace constant with a relative error at most 0.01
in all cases using several thousands of degrees of freedom.
The particular value of the trace constant $\CT$ depends considerably
on $\tilde a$. The values for $\tilde a=0.001$ and $\tilde a=1000$ differ more than
thirty times. Further, notice that the exact value of the trace constant
for $\tilde a=1$ is $\CT= (2/(\pi\coth \pi))^{1/2} \approx 0.7964$. 

As before, we illustrate the adaptive process for the case
$\tilde a = 0.001$. Figure~\ref{fi:resCT} presents
the convergence of bounds $\CTup$ and $\CTlow$ (left panel)
and the relative error $\RelErr$ with heuristic indicators
$D_1$, $D_2$ (right panel).
Good confidence in the validity of assumption \ttg{eq:closest}
stems from the fact that $D_1$ is several times smaller than $D_2$
in later stages of the adaptive process.
Figures~\ref{fi:meshesA}--\ref{fi:meshesB} (right) show two of the adapted meshes.

Let us note that Tables~\ref{ta:Friedres}, \ref{ta:Poinres}, and \ref{ta:traceres}
show highest numbers of degrees of freedom for $\tilde a = 0.1$
and $10$. This is probably a coincidence caused by the fact that the error
decreases in jumps after each mesh refinement.
For example, if $\NDOF$ is about $5\,000$ then the relative error is already
close to the threshold $\RelErr=0.01$. If it is slightly below the threshold,
we stop the algorithm, but if it is slightly above, we have to refine
the mesh one more time, which results in higher $\NDOF$
and also a smaller $\RelErr$.

\begin{figure}
\begin{center}
\includegraphics[width=0.3\textwidth]{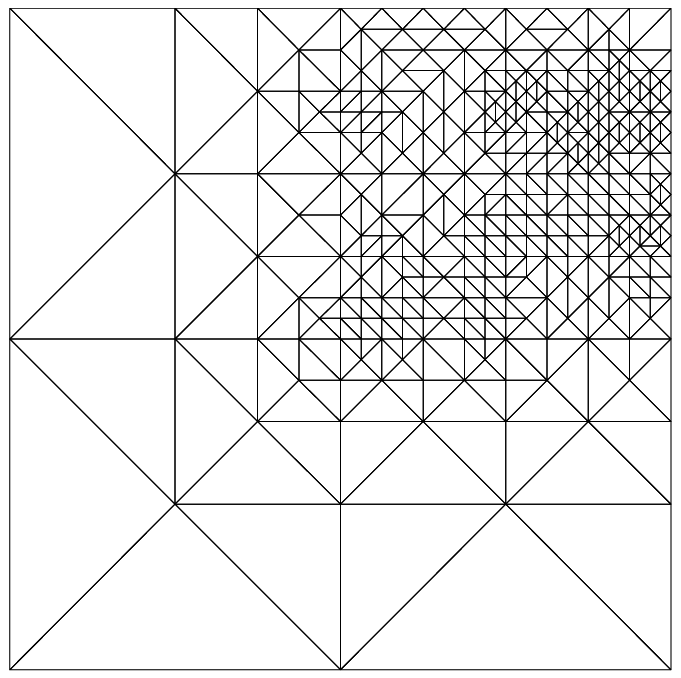}\quad
\includegraphics[width=0.3\textwidth]{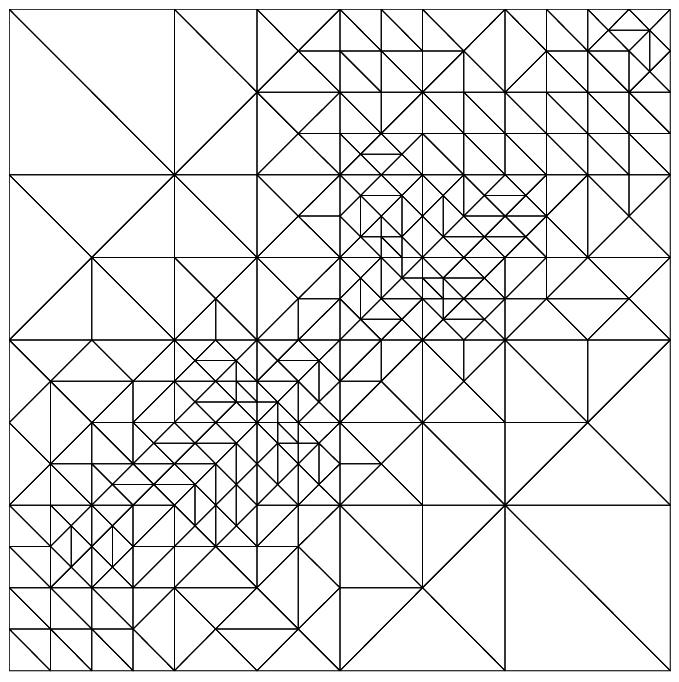}\quad
\includegraphics[width=0.3\textwidth]{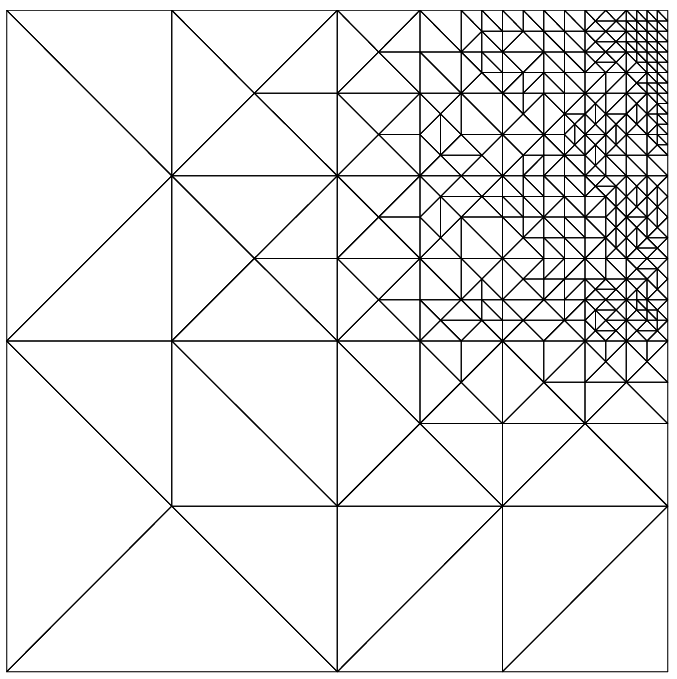}
\caption{Meshes in the middle of the adaptive process for $\tilde a = 0.001$. From left to right: Friedrichs' (adaptive step 8), Poincar\'e (adaptive step 7), trace constant (adaptive step 8).}
\label{fi:meshesA}
\end{center}
\end{figure}

\begin{figure}
\begin{center}
\includegraphics[width=0.3\textwidth]{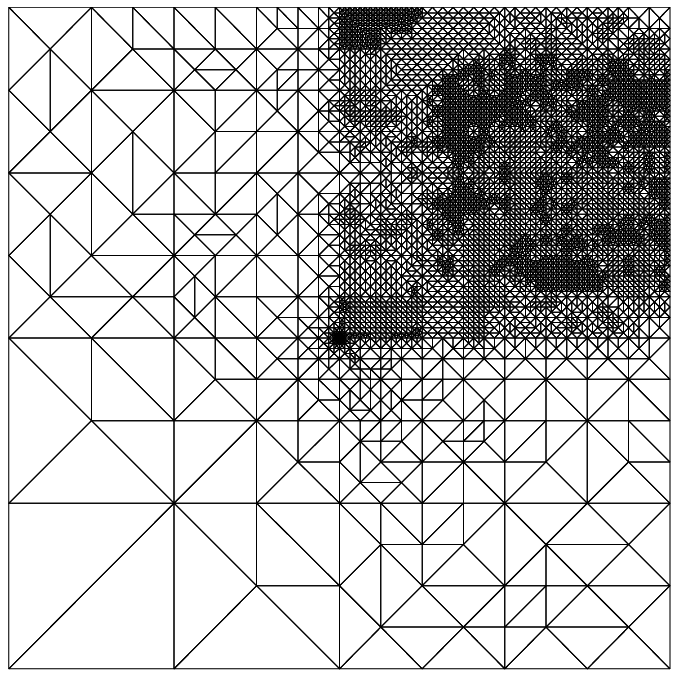}\quad
\includegraphics[width=0.3\textwidth]{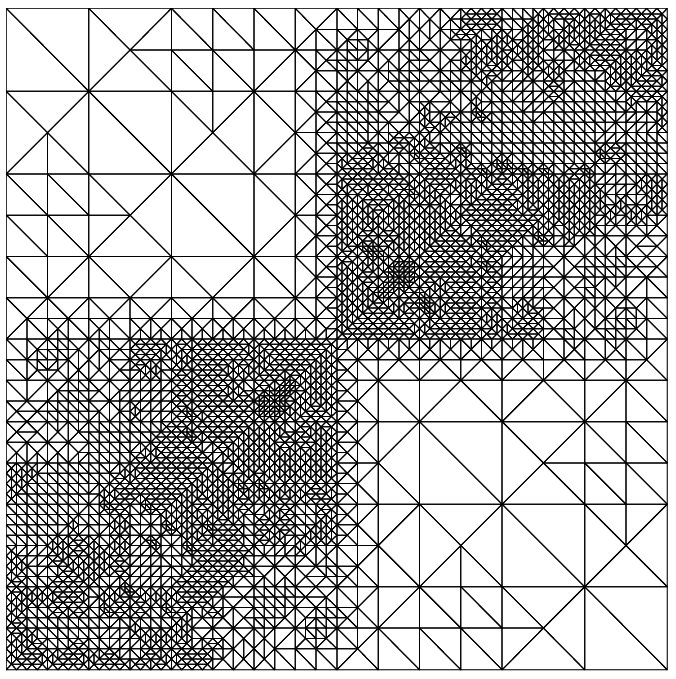}\quad
\includegraphics[width=0.3\textwidth]{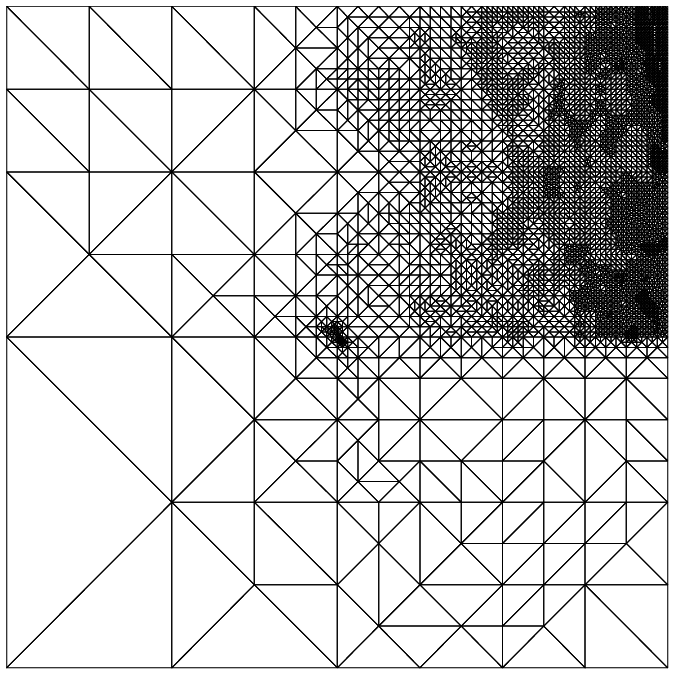}
\caption{Meshes in the final adaptive step for $\tilde a = 0.001$. From left to right: Friedrichs' (adaptive step 16), Poincar\'e (adaptive step 14), trace constant (adaptive step 15).}
\label{fi:meshesB}
\end{center}
\end{figure}

\section{Conclusions}
\label{se:concl}

We present a method for computing guaranteed lower and upper bounds
of principal eigenvalues of elliptic operators and consequently for computing guaranteed two-sided bounds of the optimal constants in Friedrichs',
Poincar\'e, trace, and similar inequalities.
The bounds are guaranteed provided there are no round-off errors and
all integrals are evaluated exactly.
Further, the bounds are guaranteed only if the domain $\Omega$ is represented exactly by used finite elements.
Furthermore, the upper bounds on eigenvalues
computed by the Galerkin method are guaranteed only if the corresponding
matrix-eigenvalue problems are solved exactly. On the other hand,
the lower bounds on eigenvalues obtained by the complementarity technique
are guaranteed even if the matrix-eigenvalue problems and
linear algebraic systems are solved approximately only.
In any case, the crucial assumption for having guaranteed lower bounds
on eigenvalues is \ttg{eq:closest}.

These two-sided bounds can be of interest if the corresponding eigenvalue
problem cannot be solved analytically and if analytical estimates
are not available or they are too inaccurate.
In particular, this is the case of
complicated geometry of the domain $\Omega$,
mixed boundary conditions,
presence of non-constant and/or anisotropic diffusion coefficient $\cA$,
presence of reaction coefficient $c$,
and presence of coefficient $\alpha$.

The method is quite general and it can be used for a wide variety of problems.
The general Hilbert space setting presented in Section~\ref{se:abseigenprob}
enables a variety of applications including linear elasticity.
We believe that this approach can be further generalized.
Nonlinear eigenvalue problems and nonsymmetric operators can be
of particular interest and generalizations in these directions
can be subject for further research.




\bibliographystyle{siam}
\bibliography{bibl}

\begin{thebibliography}{10}

\bibitem{Adams_Sob_spaces_03}
{\sc R.~A. Adams and J.~J.~F. Fournier}, {\em Sobolev spaces},
  Elsevier/Academic Press, Amsterdam, 2003.

\bibitem{robustaee:2010}
{\sc M.~Ainsworth and T.~Vejchodsk{\'y}}, {\em Fully computable robust a
  posteriori error bounds for singularly perturbed reaction--diffusion
  problems}, Numer. Math., 119 (2011), pp.~219--243.

\bibitem{Alonso_Russo_Padra_Rodri_AEE_FEM_SAVP_01}
{\sc A.~Alonso, A.~D. Russo, C.~Padra, and R.~Rodr\'\i~guez}, {\em {A
  posteriori error estimates and a local refinement strategy for a finite
  element method to solve structural-acoustic vibration problems}}, Adv.
  Comput. Math., 15 (2001), pp.~25--59.

\bibitem{AndRac:2012}
{\sc A.~Andreev and M.~Racheva}, {\em Two-sided bounds of eigenvalues}, Appl.
  Math.,  (to appear, 2013).

\bibitem{BabOsb:1989}
{\sc I.~Babu{\v{s}}ka and J.~E. Osborn}, {\em Finite element-{G}alerkin
  approximation of the eigenvalues and eigenvectors of selfadjoint problems},
  Math. Comp., 52 (1989), pp.~275--297.

\bibitem{BabOsb:1991}
\leavevmode\vrule height 2pt depth -1.6pt width 23pt, {\em Eigenvalue
  problems}, in Handbook of numerical analysis, {V}ol.\ {II}, Handb. Numer.
  Anal., II, North-Holland, Amsterdam, 1991, pp.~641--787.

\bibitem{BaiDemDonRuhVor:2000}
{\sc Z.~Bai, J.~Demmel, J.~Dongarra, A.~Ruhe, and H.~van~der Vorst}, eds., {\em
  Templates for the solution of algebraic eigenvalue problems}, vol.~11 of
  Software, Environments, and Tools, Society for Industrial and Applied
  Mathematics (SIAM), Philadelphia, PA, 2000.
\newblock A practical guide.

\bibitem{Bazley_Fox_Low_bounds_eig_Schrodinger_eq_61}
{\sc N.~W. Bazley and D.~W. Fox}, {\em {Lower bounds for eigenvalues of
  Schr\"odinger's equation}}, Phys. Rev., II. Ser., 124 (1961), pp.~483--492.

\bibitem{Beattie_Goerisch_Methods_low_bounds_eig_self-adjoint_95}
{\sc C.~Beattie and F.~Goerisch}, {\em {Methods for computing lower bounds to
  eigenvalues of self-adjoint operators}}, Numer. Math., 72 (1995),
  pp.~143--172.

\bibitem{Biegert:2009}
{\sc M.~Biegert}, {\em On traces of {S}obolev functions on the boundary of
  extension domains}, Proc. Amer. Math. Soc., 137 (2009), pp.~4169--4176.

\bibitem{Boffi:2010}
{\sc D.~Boffi}, {\em Finite element approximation of eigenvalue problems}, Acta
  Numer., 19 (2010), pp.~1--120.

\bibitem{BraSch:2008}
{\sc D.~Braess and J.~Sch{\"o}berl}, {\em Equilibrated residual error estimator
  for edge elements}, Math. Comp., 77 (2008), pp.~651--672.

\bibitem{BreFor:1991}
{\sc F.~Brezzi and M.~Fortin}, {\em Mixed and hybrid finite element methods},
  vol.~15 of Springer Series in Computational Mathematics, Springer-Verlag, New
  York, 1991.

\bibitem{CheFucPriVoh:2009}
{\sc I.~Cheddadi, R.~Fu{\v{c}}{\'{\i}}k, M.~I. Prieto, and M.~Vohral{\'{\i}}k},
  {\em Guaranteed and robust a posteriori error estimates for singularly
  perturbed reaction--diffusion problems}, M2AN Math. Model. Numer. Anal., 43
  (2009), pp.~867--888.

\bibitem{Chu_King_Tsung_Palindr_EP_sum_2010}
{\sc E.~K.-w. Chu, T.-M. Huang, W.-W. Lin, and C.-T. Wu}, {\em {Palindromic
  eigenvalue problems: A brief survey}},  (2010).

\bibitem{Fox_Rheinboldt_Comput_meth_deter_low_bounds_eig_oper_Hilbert_66}
{\sc D.~W. Fox and W.~C. Rheinboldt}, {\em {Computational methods for
  determining lower bounds for eigenvalues of operators in Hilbert spaces}},
  SIAM Rev., 8 (1966), pp.~427--462.

\bibitem{Gaal_Lin_anal_repres_theo_73}
{\sc S.~A. Gaal}, {\em {Linear analysis and representation theory}},
  {Berlin-Heidelberg-New York, Springer-Verlag}, 1973.

\bibitem{HasHla:1976}
{\sc J.~Haslinger and I.~Hlav{\'a}{\v{c}}ek}, {\em Convergence of a finite
  element method based on the dual variational formulation}, Apl. Mat., 21
  (1976), pp.~43--65.

\bibitem{Hu_Huang_Lin_Lower_bounds_EO_NFEMs_2011arXiv}
{\sc J.~Hu, Y.~Huang, and Q.~Lin}, {\em The lower bounds for eigenvalues of
  elliptic operators -- by nonconforming finite element methods}, ArXiv
  e-prints,  (2011).

\bibitem{Korotov:2007}
{\sc S.~Korotov}, {\em Two-sided a posteriori error estimates for linear
  elliptic problems with mixed boundary conditions}, Appl. Math., 52 (2007),
  pp.~235--249.

\bibitem{Kuf_John_Fucik_Function_spaces}
{\sc A.~Kufner, O.~John, and S.~Fu{\v{c}}{\'{\i}}k}, {\em Function spaces},
  Noordhoff International Publishing, Leyden, 1977.

\bibitem{KutSig:1978}
{\sc J.~R. Kuttler and V.~G. Sigillito}, {\em Bounding eigenvalues of elliptic
  operators}, SIAM J. Math. Anal., 9 (1978), pp.~768--778.

\bibitem{KutSig:1984}
\leavevmode\vrule height 2pt depth -1.6pt width 23pt, {\em Eigenvalues of the
  {L}aplacian in two dimensions}, SIAM Rev., 26 (1984), pp.~163--193.

\bibitem{Larson_Instab_Visco_Flows_92}
{\sc R.~G. Larson}, {\em Instabilities in viscoelastic flows}, Rheologica Acta,
  31 (1992), pp.~213--263.

\bibitem{LehSorYan:1998}
{\sc R.~B. Lehoucq, D.~C. Sorensen, and C.~Yang}, {\em A{RPACK} users' guide},
  Society for Industrial and Applied Mathematics (SIAM), Philadelphia, PA,
  1998.

\bibitem{Leissa_Vibrat_Plates-69}
{\sc A.~W. Leissa}, {\em {Vibration of Plates}}, NASA, Washington D.C., 1969.

\bibitem{LinXieLuoLiYan:2010}
{\sc Q.~Lin, H.~Xie, F.~Luo, Y.~Li, and Y.~Yang}, {\em Stokes eigenvalue
  approximations from below with nonconforming mixed finite element methods},
  Math. Pract. Theory, 40 (2010), pp.~157--168.

\bibitem{LinXieXu:MatComp}
{\sc Q.~Lin, H.~Xie, and J.~Xu}, {\em Lower bounds of the discretization error
  for piecewise polynomials}, Math. Comp., in press.

\bibitem{LuoLinXie:2012}
{\sc F.~Luo, Q.~Lin, and H.~Xie}, {\em Computing the lower and upper bounds of
  {L}aplace eigenvalue problem: by combining conforming and nonconforming
  finite element methods}, Sci. China Math., 55 (2012), pp.~1069--1082.

\bibitem{Martin_Electr_struc_bas_th_prac_methods_04}
{\sc R.~M. Martin}, {\em {Electronic structure. Basic theory and practical
  methods.}}, {Cambridge, Cambridge University Press}, 2004.

\bibitem{Ogawa_Protter_A_low_bound_first_eig_sec_order_ellip_oper_68}
{\sc H.~Ogawa and M.~H. Protter}, {\em A lower bound for the first eigenvalue
  of second order elliptic operators}, Proc. Amer. Math. Soc., 19 (1968),
  pp.~292--295.

\bibitem{Rannacher:1979}
{\sc R.~Rannacher}, {\em Nonconforming finite element methods for eigenvalue
  problems in linear plate theory}, Numer. Math., 33 (1979), pp.~23--42.

\bibitem{Rep:2008}
{\sc S.~Repin}, {\em A posteriori estimates for partial differential
  equations}, Walter de Gruyter GmbH \& Co. KG, Berlin, 2008.

\bibitem{Rudin_FA_91}
{\sc W.~Rudin}, {\em {Functional analysis. 2nd ed.}}, {New York, McGraw-Hill},
  1991.

\bibitem{Sigillito:1977}
{\sc V.~G. Sigillito}, {\em Explicit a priori inequalities with applications to
  boundary value problems}, Pitman Publishing, London-San Francisco,
  Calif.-Melbourne, 1977.

\bibitem{Strang_Fix_analysis_FEM_08}
{\sc G.~Strang and G.~J. Fix}, {\em {An analysis of the finite element methods.
  2nd ed.}}, Wellesley, MA: Wellesley-Cambridge Press, 2nd ed.~ed., 2008.

\bibitem{LocalHypercycle:2006}
{\sc T.~Vejchodsk{\'y}}, {\em Guaranteed and locally computable a posteriori
  error estimate}, IMA J.~Numer. Anal., 26 (2006), pp.~525--540.

\bibitem{Complement:2010}
{\sc T.~Vejchodsk{\'y}}, {\em Complementarity based a~posteriori error
  estimates and their properties}, Math. Comput. Simulation, 82 (2012),
  pp.~2033--2046.

\bibitem{AM2012}
{\sc T.~Vejchodsk{\'y}}, {\em Computing upper bounds on {F}riedrichs'
  constant}, in Applications of Mathematics 2012, J.~Brandts, J.~Chleboun,
  S.~Korotov, K.~Segeth, J.~{\v{S}}{\'{\i}}stek, and T.~Vejchodsk{\'y}, eds.,
  Institute of Mathematics, ASCR, Prague, 2012, pp.~278--289.

\bibitem{Vohralik:2011}
{\sc M.~Vohral{\'{\i}}k}, {\em Guaranteed and fully robust a posteriori error
  estimates for conforming discretizations of diffusion problems with
  discontinuous coefficients}, J. Sci. Comput., 46 (2011), pp.~397--438.

\bibitem{Wein_Stenger_Methods_Intermed_prob_eig_72}
{\sc A.~Weinstein and W.~Stenger}, {\em {Methods of intermediate problems for
  eigenvalues. Theory and ramifications}}, {New York--London, Academic Press},
  1972.

\bibitem{YanZhaLin:2010}
{\sc Y.~Yang, Z.~Zhang, and F.~Lin}, {\em Eigenvalue approximation from below
  using non-conforming finite elements}, Sci. China Math., 53 (2010),
  pp.~137--150.

\bibitem{Yosida_FA_94}
{\sc K.~Yosida}, {\em {Functional analysis. Repr. of the 6th ed.}}, {Berlin,
  Springer-Verlag}, 1994.

\end{thebibliography}

\end{document}